\documentclass[spanish,12pt,a4paper]{article}
\usepackage[T1]{fontenc}
\usepackage[width=14cm, height=22cm]{geometry}
\usepackage{mathtools}
\usepackage{amssymb}
\usepackage{amsthm}

\usepackage{hyperref}
\usepackage{mathrsfs}
\usepackage{cancel}
\usepackage[usenames,dvipsnames,table]{xcolor}
\usepackage{amsmath,amsthm,amssymb}
\usepackage{amsfonts}
\usepackage{multicol}
\usepackage{multirow}
\usepackage{tabulary}
\usepackage{polynom}
\usepackage{moreverb} 
\usepackage{float}
\usepackage{graphicx}
\usepackage[all]{xy}
\usepackage{nicefrac}

\usepackage[shortlabels]{enumitem}

\usepackage{mathpazo}
\usepackage{domitian}

\newtheorem{defi}{{\it \textbf{Definition}}}[section]
\newtheorem{teo}[defi]{Theorem}
\newtheorem{prop}[defi]{Proposition}
\newtheorem{coro}[defi]{Corollary}
\newtheorem{lema}[defi]{Lemma}
\theoremstyle{definition}
\newtheorem{nota}[defi]{Remark}

\newcommand{\ac}{{\rm ($ \ast$C) }}
\newcommand{\acc}{{\rm ($ \ast$2C) }}
\newcommand{\accc}{{\rm ($ \ast$3C) }}

\DeclareMathOperator{\id}{Id}
\DeclareMathOperator{\dx}{dx}

\title{Transference of ergodic properties of operators via matrix actions }
\author{Daniel Santacreu\thanks{Departament d'Anàlisi Matemàtica. Facultat de Matemàtiques. Universitat de València. Av. Vicent Andrés Estellés, 19. Burjassot 46100. Spain. Daniel.Santacreu@uv.es} 
\and 
Pablo Sevilla-Peris\thanks{Insitut Universitari de Matemàtica Pura i Aplicada. Universitat Politècnica de València. cmno Vera sn. 46022 València. Spain. psevilla@mat.upv.es}
}
\date{}

\begin{document}
\maketitle

\begin{abstract}
For operators defined on locally convex spaces we define the notions of \textit{boundedness} and \textit{ergodicity} associated to an infinite matrix.
Given two matrices $ A$ and $ B$, we study when $ A$-bounded operators are $ B$-ergodic. Using this approach we obtain equivalent formulations for the classical notions of power boundedness, Cesàro boundedness and mean ergodicity.

\end{abstract}

\footnotetext{Both authors were supported by grant GVA CIAICO/2023/242 and grant PID2021-122126NB-C33 funded by MICIU/AEI/10.13039/501100011033 and by ERDF/EU}

\footnotetext{\textit{AMS classiciation 2020:} 47A35, 47B33, 47B48, }

\section{Introduction}

The ergodic theory of operators deals with the problem of the behaviour of the iteration of the composition of an operator with itself. More precisely, if $ E$ is a Banach space (or, more generally, a locally convex Hausdorff space) and $ T: E \to E$ is a (continuous, linear) operator, we consider $ T^{0} = \id$, and $T^{n} = T^{n-1} \circ T$ for $ n \geq 1$. Then, the general question is to analyse the behaviour of $ (T^{n})_{n}$ as $ n$ grows, and to find conditions that ensure that the sequence converges in some sense. To this purpose, the Cesàro means of the operator are considered, defined as
\begin{equation*}
T_{[n]} = \frac{1}{n+1} \sum_{k=0}^{n} T^{k} \,.
\end{equation*}
and say that the operator is \textit{mean ergodic} if the sequence $ (T_{[n]})_{n}$ converges pointwise to some operator on $ E$. Also, the operator is \textit{power bounded} if $ \{ \Vert T^{n} \Vert \}_{n}$ is bounded (in the Banach case, if $ \{T^{n}\}_{n}$ is equicontinuous in the locally convex case). The classical theorems of Dunford-Schwartz \cite[Chapter~8]{dunfordschwartz_88} or Yosida \cite{yosida1938} give conditions so that a power bounded operator is mean ergodic. This seek of connections between `boundedness' and `convergence' properties is a constant along the theory, and there is a number of significant results in this aspect. \\

Finding alternative notions of convergence of sequences or series (such as Cesàro or Abel convergence) like, in a rather elementary level, for Fourier series, is by now a standard strategy. Toeplitz \cite{Toeplitz1911} was the first one to study systematically these methods of summation, using infinite matrices. This point of view is now somewhat classical (see \cite{cooke_65, hardy_49}). To our best knowledge, this approach through infinite matrices has not been considered to study ergodic properties of operators. In this note we propose this point of view, defining `boundedness' and `convergence' properties induced by infinite matrices, and finding conditions that allow to relate them, in the spirit of the classical results relating power and Cesàro boundedness with mean ergodicity.

\section{Matrix boundedness and ergodicity}

We fix now our general framework and establish some basic notation. In this note $ A$ will always denote an infinite lower triangular matrix That is, $A=(a_{nk})_{n,k=0}^\infty$ where $ a_{nk} \in \mathbb{C}$, with $a_{nk}=0$ whenever $i>j$. We say that such a matrix is a \textit{probability matrix} if, moreover, all elements are non-negative and the sum of each row is $ 1$. That is
\begin{enumerate}[(i)]
	\item $a_{nk}\geq 0$ for every $n,k \in \mathbb{N}_0$.
	
	\item $ \displaystyle \sum_{i=0}^{n}a_{ni}=1 $ for each $ n\in \mathbb{N}_0$.
\end{enumerate}

Also, $ \mathcal{T}$ will always denote a sequence sequence of continuous linear operators on some locally convex Hausdorff space $ E$. That is, $ \mathcal{T}=\{T_{n} \colon n \in \mathbb{N}_{0}\}$, where $ T_{n} : E \to E$ is linear and continuous for every $ n \in \mathbb{N}_{0}$. 
 We take a lower triangular matrix $ A$, and a sequence  $ \mathcal{T}$ as above and, for each $ n \in \mathbb{N}_{0}$ we consider
\begin{equation} \label{panadera}
	(A \mathcal{T})_{n} = \sum_{i=0}^{n} a_{ni} T_{i} \,.
\end{equation}
and define the \textit{$ A$-expected value} of $ \mathcal{T}$ as
\begin{equation*}
	A \mathcal{T}= \{ (A \mathcal{T})_{n} \colon n \in \mathbb{N}_{0} \}
\end{equation*}
With this notation, we say that $ \mathcal{T}$ is
\begin{itemize}
	\item $ A$-bounded if $A \mathcal{T}$ is equicontinuous, that is, for every continuous seminorm $ p$ there exists a continuous seminorm $ q$ so that
	\begin{equation} \label{abdd}
		p \Big( \sum_{i=0}^{n} a_{ni} T_{i}x \Big) \leq q(x)
	\end{equation}
	for every $ x \in E$ and all $ n \in \mathbb{N}_{0}$.
	
	\item absolutely $ A$-bounded if for every continuous seminorm $ p$ there exists a continuous seminorm $ q$ so that
	\begin{equation*}
		\sum_{i=0}^{n} p \big(a_{ni} T_{i}x \big) \leq q(x)
	\end{equation*}
	for every $ x \in E$ and all $ n \in \mathbb{N}_{0}$.
	
	\item  $ A$-ergodic if there is some operator $P$ such that $ (A \mathcal{T})_{n}x$ converges to $Px$ for every $ x \in E$.
	
	\item uniformly $ A$-ergodic if $ A \mathcal{T}$ converges uniformly on bounded sets to some operator.
	
	\item $ A$-null if $ (A \mathcal{T})_{n}x \to 0$ for every $ x \in E$.
\end{itemize}

Observe that if $\mathcal{T}$ is $A$-null then it is $A$-ergodic. Also if $E$ is barrelled, then every $A$-ergodic $\mathcal{T}$ is $A$-bounded. Let us also point out that, if $ E$ is a Banach space, then $ \mathcal{T}$ is $ A$-bounded if and only if the sequence $ \big(\Vert (A \mathcal{T})_{n} \Vert \big)_{n}$ is bounded.

\begin{nota} \label{urbanlights}
	Given a single operator $ T$, we consider the sequence of its powers, $ (T^{n})_{n}$, and we say that the operator $ T$ is 
	$ A$-bounded, 
	$ A$-ergodic or $ A$-null if the sequence $ (T^{n})_{n}$ is so. In this case we will write $ (AT)_{n}$ in \eqref{panadera}.
\end{nota}

Let us rewrite some classical ergodic properties which, in some sense are the starting point of our research, in these matrix terms. 

\begin{nota} \label{identitat}
	Let $ I=(\delta_{nk})_{n,k=0}^\infty$ be the identity matrix. Then, being $ I$-bounded is exactly being power bounded. Also, an operator $T$ is $I$-ergodic if and only if $ (T^{n})_{n}$ is pointwise convergent, and uniformly  $I$-ergodic if and only if $ (T^{n})_{n}$ converges uniformly on bounded sets.
\end{nota}

\begin{nota} \label{cesaro}
	Let us recall (see e.g. \cite{boos}), that the Cesàro mean can also be formulated in terms of matrices. This of course applies also in our setting, and (absolute) Cesàro boundedness and (uniform) mean ergodicity of operators can be described with matrices. To do so, just consider the matrix $ M=(m_{nk})_{n,k=0}^\infty$ with $m_{nk}=\frac{1}{n+1}$ when $n\leq k$ and $0$ elsewhere, that is
	\begin{equation} \label{sagrario}
		M= 
		\begin{pmatrix}
			1 & 0&0 &0 & \ldots \\
			\nicefrac{1}{2} &  \nicefrac{1}{2}  & 0 & 0 & \ldots \\
			\nicefrac{1}{3} & \nicefrac{1}{3} &  \nicefrac{1}{3}  &0 & \ldots \\
			\nicefrac{1}{4} & \nicefrac{1}{4} &  \nicefrac{1}{4}  & \nicefrac{1}{4}  & \ldots \\
			\vdots & \vdots & \vdots & \vdots & \ddots
		\end{pmatrix}
	\end{equation}
	This matrix does the job, and being (absolutely) Cesàro bounded and being (absolutely) $M$-bounded are equivalent, and (uniform) mean ergodicity is equivalent to (uniform) $ M$-ergodicity.
\end{nota}

\section{Pairs of matrices}

Our first aim is to see under which conditions these properties transfer from one matrix to another. To be more specific, given two lower triangular matrices $ A$ and $ B$, we are interested in seeing what sort of relationship between them gives that every $ A$-bounded/ergodic/null sequence is $ B$-bounded/er\-godic/null. The general idea here is, given two matrices $ A$ and $ B$, try to find some matrix $ C$  with $ CA=B$, and give conditions on $ C$ so that $ A$-properties transfer to $ B$-properties. 

\begin{nota} \label{fox240}
	Let us in first place note that $A \mathcal{T}= \{ (A \mathcal{T})_{n} \colon n \in \mathbb{N}_0 \}$ is again a sequence of continuous linear operators. Then, if $ C$ is another lower triangular matrix, straightforward computations show that $ \big( C(A \mathcal{T}) \big)_{n}  = \big((CA)  \mathcal{T}) \big)_{n}$ for every $ n$, where $ CA$ is the usual matrix product. Observe also that, if both $ A$ and $ C $ are probability matrices, then so also is $ CA$.
\end{nota}

Suppose that $ A$ and $ B$ are two lower triangular matrices for which there is some lower triangular matrix $ C=(c_{ij})$ such that $ CA=B$. Then we say that the pair $ (A,B)$ satisfies
\begin{itemize}
	\item property \ac if the sums of the rows of $ C$ are bounded, that is
	\begin{equation} \label{astc}
		K_{C}=\sup_{n \in \mathbb{N}_0} \sum_{i=0}^{\infty} \vert c_{ni} \vert < \infty
	\end{equation}
	
	\item property \acc if it satisfies \ac and, moreover, the columns of $ C$ tend to $ 0$, that is
	\begin{equation} \label{ast2c}
		\lim_{n} c_{ni} = 0\, \text{ for each fixed } i \in \mathbb{N}_0
	\end{equation}
	
	\item property \accc if the sums of the rows of $ C$ tend to 0, that is
	\begin{equation} \label{ast3c}
		\lim_{n} \sum_{i=0}^{\infty} \vert c_{ni} \vert =0 \,.
	\end{equation}
\end{itemize}

These relations between matrices allow to transfer the properties of sequences of operators from one another

\begin{prop} \label{schubert}
	Let $ (A,B)$ be a pair of lower triangular matrices. If the pair satisfies
	\begin{enumerate}[(i)]
		\item \label{madruga} \ac\!, then every $ A$-bounded $ \mathcal{T}$ is $ B$-bounded.
		
		\item  \label{null to null} \acc\!, then every $ A$-null $ \mathcal{T}$ is $ B$-null.
		
		\item  \label{bdd to null} \accc\!, then every $ A$-bounded $ \mathcal{T}$ is $ B$-null.
	\end{enumerate}
\end{prop}
\begin{proof}
	We fix a matrix $ C$ satisfying $ CA=B$, a sequence of continuous linear operators $ \mathcal{T}=(T_{n})_{n}$, and $ p$ a seminorm on $ E$. For \ref{madruga} we assume that $ \mathcal{T}$ is $ A$-bounded, and that $ C$ satisfies \eqref{astc}. Then we can find a a seminorm $ q$ satisfying \eqref{abdd} and, taking Remark~\ref{fox240} into account, we get
	\begin{multline*}
		p\big((B \mathcal{T} )_{n} x \big) = p\big((C(A \mathcal{T} ))_{n} x \big) \\
		= p \Big( \sum_{j=1}^{n} c_{nj}\sum_{i=1}^{j} a_{ij} T_{j} x \Big) 
		\leq \sum_{j=1}^{n} \vert c_{nj} \vert  \, \, p \Big( \sum_{i=1}^{j} a_{ij} T_{j} x \Big)
		\leq K_{C} q(x) \,,
	\end{multline*}
	for every $ x$ and $ n$. \\
	To prove \ref{null to null} we assume that $ \mathcal{T}$ is $ A$-null, and $ C$ satisfies \eqref{astc} and \eqref{ast2c}. Then, given $ x \in E$ we can find $ N \in \mathbb{N}_0$ so that
	\begin{equation*}
		p \Big(\sum_{i=1}^{n} a_{ni}  T_{i}x\Big) < \frac{\varepsilon}{2 K_C}
	\end{equation*}
	for every $n\geq N$. Taking into account that $ a_{ji}=0$ whenever $ i >j$,  we have
	\begin{equation*}
		(B\mathcal{T})_{n} x 
		= \sum_{i=1}^{n} \sum_{j=1}^{n} c_{nj} a_{ji} T_{i}x 
		=  \sum_{j=1}^{n} c_{nj}  \sum_{i=1}^{n}a_{ji} T_{i}x
		=  \sum_{j=1}^{n} c_{nj}  \sum_{i=1}^{j}a_{ji} T_{i}x \,.
	\end{equation*}
	If $ n \geq N$ then
	\begin{multline*}
		p \big(  (B\mathcal{T})_{n} x \big) 
		\leq p\Big(\sum_{j=1}^{N} c_{nj}  \sum_{i=1}^{j}a_{ji}T_{i}x\Big)
		+p\Big(\sum_{j=N+1}^{n} c_{nj}  \sum_{i=1}^{j}a_{ji} T_{i}x\Big) \\
		\leq \max_{1 \leq j \leq N} p\big( (A\mathcal{T})_{j} x\big)\sum_{j=1}^{N} \vert c_{nj} \vert
		+ \frac{\varepsilon}{2 K_{C}}\sum_{j=N+1}^{n} \vert c_{nj} \vert
		\leq K_{2} \sum_{j=1}^{N} \vert c_{nj} \vert + \frac{\varepsilon}{2} \,.
	\end{multline*}
	Now, for each $ j=1, \ldots ,N$, using \eqref{ast2c}, choose $ n_{j} \geq N$ so that $ \vert c_{nj}\vert \leq \frac{\varepsilon}{2 N K_{2}}$. Taking $ n_{0} = \max \{n_{1}, \ldots, n_{N}\}$ we have
	\begin{equation*}
		p \big(  (B\mathcal{T})_{n} x \big) < \varepsilon\,,
	\end{equation*}
	for every $ n \geq n_0$. This gives the claim.\\
	Finally, to prove that \ref{bdd to null} holds, we assume again that $ \mathcal{T}$ is $ A$-bounded, and take $p$ so that
	\begin{equation*} 
		p \Big( \sum_{i=1}^{n} a_{ni} T_{i}x \Big) \leq q(x)
	\end{equation*}
	for every $ x \in E$ and all $ n \in \mathbb{N}_{0}$. Therefore by \accc we have
	\begin{multline*}
		\lim_n p \big(  (B\mathcal{T})_{n} x \big) 
		= \lim_n p\Big(\sum_{j=1}^{n} c_{nj}  \sum_{i=1}^{j}a_{ji} T_{i}x \Big) \\
		\leq \lim_n \Big[\sum_{j=1}^{n} \vert c_{nj} \vert \, p\big( \sum_{i=1}^{j}a_{ji} T_{i}x\big)\Big]
		\leq\, q(x) \lim_n \sum_{j=1}^{n} \vert c_{nj} \vert  =0 \,,
	\end{multline*}
	for every $x\in E$, which gives the claim.
\end{proof}

We can look at these properties of matrices in terms of operators acting on $\ell_\infty$ and $c_0$. Note that, if $ C=(c_{nk})_{n,k=0}^{\infty}$ is a (lower triangular) matrix, we can define a linear mapping on $ \mathbb{C}^{\mathbb{N}_0}$ by doing
\begin{equation} \label{geminiani}
	Cx = \bigg( \sum_{i=0}^{\infty} c_{ni}x_{i} \bigg)_{n =0}^{\infty}
\end{equation}
for each sequence $ x=(x_{i})_{i=0}^{\infty}$ (note that each sum in \eqref{geminiani} is finite). Properties \ac\!, \acc and \accc can be reformulated in terms of properties of this operator. For \ac and \acc this was already done in \cite[Theorem~4.51-C]{Taylor} (see also \cite[Lemma~1]{Reade} for \acc\!). For condition \accc one can check \cite[p.~4 (21.1)]{Stieglitz}. We collect these in the following result and, for the sake of completeness, include a proof more akin to our notation and point of view.

\begin{prop}\label{cpeb}
	Let $ A,B$ and $ C$ be lower triangular matrices so that $ CA=B$. Then, the pair $ (A,B)$ satisfies 
	\begin{enumerate}[(i)]
		\item \label{cpeb1} \ac if and only if $ C: \ell_{\infty} \to \ell_{\infty}$ is continuous.
		
		\item  \label{cpeb2} \acc  if and only if $ C: c_{0} \to c_{0}$ is continuous.
		
		\item  \label{cpeb3} \accc  if and only if $ C: \ell_{\infty} \to c_0$ is compact.
	\end{enumerate}
\end{prop}
\begin{proof} 	
	\ref{cpeb1} Suppose that $ (A,B)$ satisfies \ac, and take $x\in \ell_{\infty}$. Then
	\begin{equation} \label{gustafsson}
		\Vert Cx\Vert_{\infty}=\sup_n \left| \sum_{i=0}^{\infty} c_{ni}x_{i}  \right|\leq \Vert x\Vert_{\infty} \sum_{i=0}^{\infty} |c_{ni}|\leq K_{C}\Vert x\Vert_{\infty} \,,
	\end{equation}
	which gives the claim. Suppose now that the operator $ C$ is continuous, and define a matrix  $D=(d_{nk})_{n,k=0}^\infty$ with $d_{nk}=\tfrac{\overline{c_{nk}}}{\vert c_{nk}\vert }$ if $c_{nk}\neq 0$ and $0$ otherwise, now for each $n$ consider
	\begin{equation}\label{zn}
		z(n) = \big(d_{n0}, d_{n1}, d_{n2}, \ldots)  \,.
	\end{equation}
	Then $ \Vert Cz(n)\Vert_{\infty}\leq \Vert C \Vert \cdot \Vert z(n) \Vert_{\infty} = \Vert C \Vert $ for every $n$, and
	\begin{equation} \label{rienzi}
		\sum_{i=0}^\infty |c_{ni}| = \left| \sum_{i=0}^\infty c_{ni}z(n)_i \right|= |(Cz(n))_n|\leq \Vert C \Vert \,, 
	\end{equation}
	for every $n$, and the pair $ (A,B)$ satisfies \ac\!.\\
	\ref{cpeb2} Assume that the pair satisfies \acc and let us see that the operator $ C: c_{0} \to c_{0}$ is continuous. Taking \ref{cpeb1}, it is only left to see that $ C(c_{0}) \subseteq c_{0}$. Take, then $x\in c_0$ and fix $\varepsilon>0$. We can find $i_0$ such that $|x_i|<\frac{\varepsilon}{2 K_{C}}$ for every $i>i_0$. Then
	\[ 
	\vert(Cx)_n\vert=\left| \sum_{i=0}^\infty c_{ni}x_i \right|\leq \sum_{i=0}^{i_0} |c_{ni}||x_i| + \sum_{i=i_0+1}^\infty |c_{ni}||x_i|\leq \Vert x\Vert_{\infty} \sum_{i=0}^{i_0} |c_{ni}| + \frac{\varepsilon}{2}\,.  
	\]
	Proceeding as in the proof of Proposition~\ref{schubert}--\ref{null to null}, we get that  $Cx\in c_0$. Now, if the operator $ C: c_{0} \to c_{0}$ is continuous, note that the sequences  defined in \eqref{zn} are in $c_0$, and the argument in \eqref{rienzi} shows that \eqref{astc} holds. We finish the argument by fixing an index $i\in\mathbb{N}_0$. For the $i$-th element of the canonical basis $e_i\in c_0$ we have
	\[ 
	0=\lim_{n\to \infty} (Ce_i)_n= \lim_{n\to \infty} c_{ni}\,, 
	\]
	and \acc is satisfied.
	
	Before we get in the proof of \ref{cpeb3}, recall that a bounded $K\subset c_{0}$ is compact if and only if there exists a  sequence of positive numbers $a=(a_n)_n \in c_0$ such that $\sup_{x\in K} |x_n|\leq a_n$, for every $n\in \mathbb{N}_0$ (see \cite[Chapter~II, p.~15]{Diestel}). Now, let $B_{\ell_{\infty}}$ denote the closed unit ball in $\ell_\infty$ . Observe that
	\[ 
	\sup_{x\in B_{\ell_{\infty}}}|(Cx)_n|\leq \sup_{x\in B_{\ell_{\infty}}} \sum_{i=0}^\infty \left|c_{ni}x_i  \right| 
	\leq \sup_{x\in B_{\ell_{\infty}}}\Vert x\Vert_{\infty} \sum_{i=0}^\infty \left|c_{ni}  \right|= \sum_{i=0}^\infty \left|c_{ni}  \right|\,. 
	\]
	Taking $a_n=\sum_{i=0}^\infty \left|c_{ni}  \right|$, \eqref{ast3c} yields  the claim. 
	Conversely, if $ C: \ell_{\infty} \to c_0$ is compact, we take the sequences  $z(n)$ defined in \eqref{zn}, and there is $a\in c_0$ such that $\sup_{k\in\mathbb{N}_0}|(Cz(k))_n| \leq a_n$ for every $n$. In particular 
	\[ 
	\lim_{n\to \infty} \sum_{i=0}^\infty |c_{ni}|= \lim_{n\to \infty} |(Cz(n))_n|\leq \lim_{n\to \infty} a_n =0\,, 
	\]
	and \accc follows.
\end{proof}

Observe that the matrix $M$ acting on $\mathbb{C}^{\mathbb{N}_0}$ (see \eqref{sagrario}) coincides with the Cesàro operator. This gives that the properties of the pair $(\id,M)$ are connected to the properties of the Cesàro operator. The Cesàro operator is continuous when acting on $\ell_\infty$ or $c_0$ but it is not compact (see \cite{AlbaneseBonetRicker15}).\\

Now that we have described the matrix properties in terms of operators, and bearing in mind that the product of matrices translates into composition of operators, we immediately have the following.

\begin{coro}
	Let $ A$, $B$ and $D$ be lower triangular matrices. Then the following are satisfied
	\begin{enumerate}[(i)]
		\item If the pairs $ (A,B)$ and $(B,D)$ satisfy \ac then the pair $(A,D)$ also satisfies \ac\!.
		\item If the pairs $ (A,B)$ and $(B,D)$ satisfy \acc then the pair $(A,D)$ also satisfies \acc\!.
		\item If the pair $(A,B)$ satisfies \ac or  \accc  and the pair $(B,D)$ satisfies \accc then the pair $(A,D)$ also satisfies \accc\!. 
		\item If the pair $(A,B)$ satisfies \accc and the pair $(B,D)$ satisfies \acc then, the pair $(A,D)$ also satisfies \accc\!.
	\end{enumerate}
\end{coro}

If the matrix $ A$ is invertible, then we can go further, giving more complete descriptions of these properties (see Theorems~\ref{cohen}, \ref{pato} and \ref{leonard}). In this case the only possible $ C$ with $ CA=B$ is $ C=B A^{-1}$.
On the other hand, $ A$ is invertible if and only if $ a_{nn} \neq 0$ for every $ n \in \mathbb{N}_0$. Then each $ n \times n$ block (let us denote these by $ A_{n \times n}$) is invertible. An elementary computation shows that, if $ d_{ij}$ and $ d_{ij}^{n}$ denote the elements of $ A^{-1}$ and $ A_{n \times n}^{-1}$ respectively, then $ d_{ij} = d_{ij}^{n}$ for every $1 \leq i,j \leq n$. In other words, the inverse of $ A$ can be computed blockwise by computing the inverses of the $ n \times n$ blocks. \\

We recall here the definition of the backward shift $ S : \ell_{\infty} \to \ell_{\infty}$, given by
$ S(x_{1},x_{2},x_{3}, \ldots) = (x_{2},x_{3},\ldots)$. Clearly $ \Vert S \Vert = 1$, and it is power bounded.

\begin{teo} \label{cohen}
	Let $ A$ and $ B$ be two lower triangular matrices, being $ A$ invertible. Then the following are equivalent
	\begin{enumerate}[(i)]
		\item \label{cohen1} the pair $ (A,B)$ satisfies \ac\!.
		
		\item \label{cohen2} if $ E$ is any lcHs, then  every $ A$-bounded sequence of operators on $ E$ is $ B$-bounded. 
		
		\item \label{cohen3} every $ A$-bounded sequence of operators on $ \ell_{\infty}$ is $ B$-bounded. 
		
		\item \label{cohen4} the sequence $\mathcal{S}= \big\{ (A^{-1} S)_{n} \colon  n=0,1,2,3, \ldots \big\} $ of operators on $ \ell_{\infty}$ is $ B$-bounded. 
		
	\end{enumerate} 
\end{teo}
\begin{proof}
	Proposition~\ref{schubert}--\ref{madruga} is exactly that~\ref{cohen1} implies~\ref{cohen2}, and~\ref{cohen2} clearly implies~\ref{cohen3}. Let us suppose now that every $ A$-bounded sequence of operators $ \mathcal{T}=(T_{n})_{n}$ on $ \ell_{\infty}$ is $ B$-bounded. Note that $ \big(A(A^{-1}\mathcal{S})\big)_{n} = \big(AA^{-1}(\mathcal{S})\big)_{n} = S^{n}$ for every $ n$. Since  $ S$ is power bounded. this gives that $\mathcal{S} $ is $ A$-bounded and therefore it is $ B$-bounded.\\
	To complete the proof, assume that~\ref{cohen4} holds, and let us see that the pair $ (A,B)$ satisfies \ac\!. Define $ C=B A^{-1}$, and we need to check that it satisfies \eqref{astc}. For each $ n \in \mathbb{N}_0$ we consider the sequences $z(n)$ defined in \eqref{zn}. Note that
	\begin{equation*}
		(B(A^{-1} S))_n = (C S)_{n} = \sum_{i=0}^{n} c_{ni} S^{i} \,,
	\end{equation*}
	and, by assumption $ \big(\Vert (B(A^{-1} S))_n \Vert\big)_{n}$ is bounded (say, by $ K$). Then we have,  for every $ n \in \mathbb{N}_{0}$,
	\begin{multline*}
		K \geq \Vert (B(A^{-1} S))_n \Vert 
		= \sup_{\Vert x \Vert_{\infty} \leq 1} \Big\Vert \sum_{i=0}^{n} c_{ni} S^{i} x \Big\Vert_{\infty}
		\geq \Big\Vert \sum_{i=0}^{n} c_{ni} S^{i} z(n) \Big\Vert_{\infty}\\
		= \sup_{j \in \mathbb{N}_0}  \Big\vert \sum_{i=0}^{n} c_{ni} \frac{\overline{c_{n(i+j)}}}{\vert c_{n(i+j)}\vert } \Big\vert
		\geq \Big\vert \sum_{i=0}^{n} c_{ni}  \frac{\overline{c_{ni}}}{\vert c_{ni}\vert } \Big\vert
		= \Big\vert \sum_{i=0}^{n} \vert c_{ni}\vert  \Big\vert
		=  \sum_{i=0}^{n} \vert c_{ni}\vert  \,,
	\end{multline*}
	which proves our claim.
\end{proof}

To deal with \acc we need a different operator. For $ x = (x_{n})_{n} \in c_{0}$ we write  
$\langle e_n,x\rangle = x_{n}$, and consider the sequence of operators on $ c_{0}$ given by
\begin{equation*}
	\mathcal{E}= \big\{ \langle e_n,\cdot\rangle e_n \colon  n=0,1,2,3, \ldots \big\} \,.
\end{equation*}

\begin{teo} \label{pato}
	Let $ A$ and $ B$ be two lower triangular matrices, being $ A$ invertible. Then the following are equivalent
	\begin{enumerate}[(i)]
		\item \label{pato1} the pair $ (A,B)$ satisfies \acc\!.
		
		\item \label{pato2} if $ E$ is any lcHs,  then  every $ A$-null sequence of operators on $ E$  is $ B$-null.
		
		\item \label{pato3} every $ A$-null sequence of operators on $ c_0$ is $ B$-null. 
		
		\item \label{pato4} the sequence $A^{-1} \mathcal{E}$ of operators on $ c_0$ is $ B$-null. 
	\end{enumerate} 
\end{teo}
\begin{proof}
	Proposition~\ref{schubert}--\ref{null to null} is exactly that~\ref{pato1} implies~\ref{pato2}, and~\ref{pato2} clearly implies~\ref{pato3}. Let us suppose now that every $ A$-null sequence of operators on $ c_0$ is $ B$-null. For each $ x \in c_{0}$ we have
	\begin{equation*}
		A(A^{-1} \mathcal{E})_{n}(x) =\mathcal{E}_{n}x =  \langle e_n, x \rangle e_n = x_{n} e_{n} \,,
	\end{equation*}
	and 
	\begin{equation*}
		\lim_{n}  \big\Vert A(A^{-1} \mathcal{E})_{n}(x) \big\Vert_{\infty} = \lim_{n} \vert x_{n} \vert =0\,.
	\end{equation*}
	So, $A^{-1} \mathcal{E}$ is $ A$-null and, by assumption, also $ B$-null.
	Finally, assuming that~\ref{pato4} holds, let us see that the pair $ (A,B)$ satisfies \acc\!. We need to check that $ C=B A^{-1}$ satisfies \eqref{astc} and \eqref{ast2c}. For the first condition, take again the sequences  $z(n)$ defined in \eqref{zn}. Since $A^{-1} \mathcal{E}$ is $ B$-null, it is also $ B$-bounded, and we can find $ K>0$ such that, for every $ n \in \mathbb{N}_0$, we have
	\begin{multline*}
		K \geq \Vert (BA^{-1} \mathcal{E})_n \Vert = \sup_{\Vert x \Vert_{\infty} \leq 1} \Big\Vert \sum_{i=0}^{n} c_{ni} \langle e_i,x\rangle e_i \Big\Vert_{\infty} \\
		\geq \Big\Vert \sum_{i=0}^{n} c_{ni} \langle e_i,z(n)\rangle e_i \Big\Vert_{\infty}
		=   \Big\vert \sum_{i=0}^{n} c_{ni} \frac{\overline{c_{ni}}}{\vert c_{ni}\vert }  \Big\vert
		=  \sum_{i=0}^{n} \vert c_{ni}\vert  \,,
	\end{multline*}
	which proves that \eqref{astc} holds.\\	
	On the other hand, we know that $  \big\{  (B(A^{-1} \mathcal{E}))_n \colon n=0,1,2,3, \ldots \big\} $ converges to $0$ pointwise. In particular, for each fixed $e_i$ we have
	\begin{multline*}
		0=\lim_{n\to\infty}\Vert (B(A^{-1} \mathcal{E}))_n e_i \Vert_\infty = \lim_{n\to\infty} \Big\Vert \sum_{j=0}^{n} c_{nj} \langle e_j,e_i\rangle e_j \Big\Vert_{\infty}\\
		= \lim_{n\to\infty}  \Big\vert \sum_{i=0}^{n} c_{ni} \delta_{ji} \Big\vert 
		=\lim_{n\to\infty} \vert c_{ni}\vert  \,,
	\end{multline*}
	which gives \eqref{ast2c} and proves our claim.
\end{proof}

\begin{teo} \label{leonard}
	Let $ A$ and $ B$ be two lower triangular matrices, being $ A$ invertible. Then the following are equivalent
	\begin{enumerate}[(i)]
		\item \label{leonard1} the pair $ (A,B)$ satisfies \accc\!.
		
		\item \label{leonard2} if $ E$ is any lcHs,  then  every $ A$-bounded sequence of operators on $ E$  is $ B$-null.
		
		\item \label{leonard3} every $ A$-bounded sequence of operators on $ \ell_{\infty}$ is $ B$-null. 
		
		\item \label{leonard4} the sequence $\mathcal{S}= \big\{ (A^{-1} S)_{n} \colon  n=0,1,2,3, \ldots \big\} $ of operators on $ \ell_{\infty}$ is $ B$-null. 
	\end{enumerate} 
\end{teo}
\begin{proof}
	As in the proof of Proposition~\ref{cohen}, it sufices to show that \ref{leonard4} implies \ref{leonard1}. Let us see that the pair $ (A,B)$ satisfies \accc\!. We need to check that $ C=B A^{-1}$ satisfies \eqref{ast3c}. Define $D=(d_{nk})_{n,k=0}^\infty$ with $d_{nk}=\tfrac{\overline{c_{nk}}}{\vert c_{nk}\vert }$ if $c_{nk}\neq 0$ and $0$ otherwise, consider the element of $\ell_\infty$ given by 
	\begin{equation}
		z=\big(d_{00},d_{10},d_{11},\ldots,d_{n0}, \ldots, d_{nn}, \ldots) \,.
	\end{equation}
	We know that $ B\mathcal{S}= \big\{  (B(A^{-1} S))_n \colon n=0,1,2,3, \ldots \big\} $ converges to $0$ pointwise. In particular, for $z$ we have
	\begin{multline*}
		0=\lim_{n\to\infty}\Vert (B(A^{-1} S))_nz \Vert_\infty = \lim_{n\to\infty} \Big\Vert \sum_{i=0}^{n} c_{ni} S^{i} z \Big\Vert_{\infty}
		= \lim_{n\to\infty} \sup_{j \in \mathbb{N}_0}  \Big\vert \sum_{i=0}^{n} c_{ni} z_{i+j} \Big\vert \\
		\geq \lim_{n\to\infty} \Big\vert \sum_{i=0}^{n} c_{ni}  \frac{\overline{c_{ni}}}{\vert c_{ni}\vert } \Big\vert
		= \lim_{n\to\infty} \Big\vert \sum_{i=0}^{n} \vert c_{ni}\vert  \Big\vert
		=  \lim_{n\to\infty} \sum_{i=0}^{n} \vert c_{ni}\vert  \,,
	\end{multline*}
	which proves our claim.
\end{proof}

\begin{nota}
	When we want to deal with absolute boundedness, the Schur product (which we denote by $ \times$) happens to be a more convenient tool. Suppose $ D$ is a bounded matrix (in the sense that $ \sup_{i,n} \vert d_{ni} \vert < \infty$) and let $ B=D\times A$. Take now an absolutely $ A$-bounded sequence of operators $ \mathcal{T}$, and for a continuous seminorm $ p$ choose $ q$ (a continuous seminorm) satisfying \eqref{abdd}. Then
	\begin{equation*}
		\sum_{i=1}^{n} p \big(b_{ni} T_{i} x\big)
		= \sum_{i=1}^{n} p \big(d_{ni} a_{ni} T_{i} x\big)
		\leq \sup_{j,m}\vert d_{mj} \vert \sum_{i=1}^{n} p \big(a_{ni} T_{i} x\big)
		\leq K q(x) \,,
	\end{equation*}
	for every $ x$. This shows that $ \mathcal{T}$ is absolutely $ B$-bounded.
\end{nota}

So far we have established conditions linking boundedness and being null for different matrices. We are also interested in ergodicity, but this requires some more work. We give some conditions in Theorem~\ref{suppe}.

\section{Operators. Eberlein's theorem and consequences}

We focus now on the behaviour of operators (recall Remark~\ref{urbanlights}), trying to find conditions that yield that an $ A$-bounded operator is $ A$-ergodic, following the guideline of classical results relating power boundedness and mean ergodicity. From now on we will always assume that $ A$ is a probability matrix. Our first step is to get an analogue of Eberlein’s mean ergodic theorem (see \cite[Chapter~2, \S~2.1, Theorem~1.1]{krengel} or \cite[Theorem~5.4]{llibre}. Here we need some conditions that guarantee certain stability of the limit of the iterates. In the classical case the operator has to satisfy $ \frac{1}{n} T^{n} x_{0} \to 0$ for some $ x_{0}$. In our setting we ask \eqref{limit T invariant}, which in the case of the Cesàro matrix \eqref{sagrario} gives the previous condition. Before we proceed,  let us recall that, if $(X, \tau)$ is a Hausdorff topological space, then  a point $y\in E$ is a cluster point of the sequence $(x_n)_{n} \subseteq X$ if  every neighbourhood of $y$ contains infinitely many points of $(x_n)_{n}$. If $ X$ is compact, then every sequence  in $X$ has a cluster point.

\begin{teo}\label{Eberlein}
	Let $ A$ be a probability matrix, and $T$ an $A$-bounded operator. Suppose $x_{0} \in E$ satisfies 
	\begin{equation}\label{limit T invariant}
		\lim_{n\to\infty}(T- \id) (AT)_{n}x_{0}=\lim_{n\to\infty}  (AT)_{n} (T-\id)x_{0}=0 \,.
	\end{equation}
	Then, for each fixed $y\in E$, the following conditions are equivalent
	\begin{enumerate}[(i)]
		\item \label{eb1} $Ty=y$ and $y$ belongs to the closed convex hull of the set $\{T^{n}x_{0}:n\in \mathbb{N}_0\}$,
		\item\label{eb2} $y=\lim\limits_{n\to\infty}(AT)_{n}x_{0}$,
		\item\label{eb3} $y=\sigma(E,E^\prime)-\lim\limits_{n\to\infty}(AT)_{n}x_{0}$,
		\item\label{eb4} $y$ is a $\sigma(E,E^\prime)$-cluster point of $\big\{ (AT)_{n}x_{0} \colon n \in \mathbb{N}_{0}\big\}$. 
	\end{enumerate}
\end{teo}
\begin{proof}
	The implications~\ref{eb2}$\Rightarrow$\ref{eb3}$\Rightarrow$\ref{eb4} are trivial. Let us see that~\ref{eb1} implies~\ref{eb2}. Fix a continuous seminorm $ p$ and $\varepsilon>0$. Since $ T$ is $ A$-bounded, we can find a seminorm $q$ so that 
	\begin{equation} \label{mikel}
		p((AT)_{n} z)\leq q(z) \,,
	\end{equation}
	for all $z\in E$ and all $n\in \mathbb{N}_0$. Now, $y$ belongs to the closed convex hull of the set $(T^m x_{0})_{m}$, then  there are $m\in \mathbb{N}_0$ and $\alpha_0,\dots, \alpha_m \geq 0$, with $\sum_{k=0}^m\alpha_k=1$, so that 
	\begin{equation} \label{landa}
		q\Big(y - \sum_{k=0}^m \alpha_kT^kx_{0}\Big) < \frac{\varepsilon}{2}\,.
	\end{equation}
	Let us write $ Sx_{0}= \sum_{k=0}^m \alpha_kT^kx_{0}$. Observe that $(AT)_n y= \sum_{i=1}^n a_{ni} T^{i}y= \sum_{i=1}^n a_{ni}y=y$ for every $n\in \mathbb{N}_0$.
	Then, for $ n \in \mathbb{N}_0$ we have
	\begin{multline*}
		p(y-(AT)_n x_{0})=p \big((AT)_n y-(AT)_n Sx_{0}+(AT)_n Sx_{0}-(AT)_n x_{0}\big) \\
		\leq p\big((AT)_n (y-Sx_{0})\big) + p\big((AT)_n Sx_{0}-(AT)_n x_{0}\big) \,.
	\end{multline*}
	We deal with the first term using \eqref{mikel} and \eqref{landa}, getting
	\begin{equation} \label{missa}
		p\big((AT)_n (y-Sx_{0})\big) \leq q(y-Sx_{0}) < \frac{\varepsilon}{2} \,.
	\end{equation}
	To handle the second term, observe that
	\begin{multline*}
		p\big((AT)_n Sx_{0}-(AT)_n x_{0}\big)
		=  p\left(\sum^{m}_{k=0}\alpha_k (AT)_n T^kx_{0}-(AT)_n x_{0}\right) \\
		\leq \sum_{k=0}^m \alpha_k p\big( (AT)_n T^k x_{0}-(AT)_n x_{0} \big)
	\end{multline*}
	Observe that for $k,n \in \mathbb{N}_0$ we have
	\[
	(AT)_n T^k x_0- (AT)_n x_{0} = \sum_{i=0}^{k-1} T^{i}(AT)_n (T-\id)x_{0} \,,
	\]
	Fix now $k= 1, \ldots , m$ and take $ 1 \leq i \leq k$. Since  $T^i$ is continuous, we can find a seminorm $q_i$ such that $ p(T^iz)<q_i(z)$ for all $z\in E$. Now, using  \eqref{limit T invariant} we can find $n_i(k)\in \mathbb{N}_0$ such that 
	\begin{equation*}
		q_{i} \big(  (AT)_n (T-\id)x_{0} \big) \leq \frac{\varepsilon}{2  k} \,,
	\end{equation*}
	for every $ n \geq n_i(k)$. Taking $ n_{0} = \max \{ n_j(k) \colon 1 \leq j,k \leq m \}$ and $ n \geq n_{0}$ we have
	\begin{equation*}
		\sum_{i=0}^{k-1} p \big(T^{i}(AT)_n (T-\id)x_{0} \big)
		\leq \sum_{i=0}^{k-1} q_{i} \big((AT)_n (T-\id)x_{0} \big)
		< \frac{\varepsilon}{2} \,,
	\end{equation*}
	for each $ k = 1 \ldots ,m$. Hence
	\begin{equation*}
		p\big((AT)_n Sx_{0}-(AT)_n x_{0}\big)
		\leq  \sum_{k=0}^m \alpha_k  \frac{\varepsilon}{2}  = \frac{\varepsilon}{2}  \,.
	\end{equation*}
	This and \eqref{missa} show that $ p(y-(AT)_n x_{0}) < \varepsilon$ for every $ n \geq n_{0}$ and yield the conclusion.\\
	To complete the proof let us see that~\ref{eb4} implies~\ref{eb1}. Let us denote by $C$ the convex hull of the set $\{T^mx_{0}:m\in \mathbb{N}_0\}$, and observe that (because $ A$ is a probablity matrix) $\{(AT)_n x_0:n\in \mathbb{N}_0\}\subseteq C$. Since $ y$ is a $ \sigma(E,E')$-cluster point, it belongs to the $ \sigma(E,E')$-closure of $ C$ which (being $ C$ convex), coincides with the $ \Vert \cdot \Vert$-closure. So, it is only left to see that $ Ty=y$. To do that, fix $v\in E^\prime$ and $ \varepsilon >0$ and note that by  \eqref{limit T invariant} we have
	\[
	\lim_{n\to \infty} v \big(T(A T)_n x_{0} - (A T)_n x_{0} \big)=0\,,
	\]
	and we can find $n_0\in \mathbb{N}_0$  such that 
	\[ 
	\left| v\big(T(AT)_n x_{0} - (AT)_n x_{0} \big) \right| < \frac{\varepsilon}{3} \,,
	\]
	for all $n\geq n_0$. Consider now the set
	\begin{equation*}
		U:=\left\{ z\in E: |v(z-y)|<\varepsilon/3, \big| T^t(v)(z-y) \big|<\varepsilon/3 \right\} \,,
	\end{equation*} 
	which is a $\sigma(E,E^\prime)$-neighbour\-hood of $y$. Since $y$ is a $\sigma(E,E^\prime)$-cluster point of $\{T^{n}x_{0}:n\in \mathbb{N}_0\}$, given $n_0$  there is $n>n_0$ such that $\left|v(y-(AT)_n x_{0})\right|<\varepsilon/3$ and $\left|T^t(v)(y-(AT)_n x_{0})\right|<\varepsilon/3$. Now, we have
	\begin{multline*}
		\left| v(y-Ty)\right| \leq \left|v(y-(A\mathcal{T})_n x)\right|  \\+ \left|v((A\mathcal{T})_n x-T(A\mathcal{T})_n x) \right| + \left|v(T(AT)_n x-Ty) \right|  <\varepsilon.
	\end{multline*}
	Since $v\in E^\prime$ and $\varepsilon>0$ were arbitrary, the Hahn-Banach Theorem yields $y=Ty$.
\end{proof}

As a consequence of the previous result we are able to characterise $A$-ergodicity for $A$-bounded operators (cf. \cite{ABR} or \cite[Corollary~5.6]{llibre}).

\begin{coro}\label{AE coro}
	Let $A$ be probability matrix, and $ T$ an $A$-bounded operator satisfying \eqref{limit T invariant} for every $x\in E$. Then $T$ is $A$-ergodic if and only if $\{(AT)_nx:n\in \mathbb{N}_0\}$ is $\sigma(E,E^\prime)$-relatively compact for every $x\in E$.
\end{coro}
\begin{proof}
	If $T$ is $A$-ergodic the sequence $\{(AT)_nx:n\in \mathbb{N}_0\}$ is convergent in $E$ for every $x\in E$. In particular, it is a $\sigma(E,E^\prime)$-relatively compact set for each $x\in E$.
	
	Conversely, we have that given $x\in E$ the set $\{(AT)_nx:n\in \mathbb{N}_0\}$ is $\sigma(E,E^\prime)$-relatively compact, and therefore this sequence has a  $\sigma(E,E^\prime)$-cluster point $y\in E$. By Theorem~\ref{Eberlein} necessarily $y=\lim\limits_{n\to\infty}(AT)_nx$. We can define 
	\[Px:=\lim\limits_{n\to\infty}(AT)_nx\]
	for each $x\in E$. Since $((AT)_{n})_n$ is equicontinuous we have that $P\in\mathcal{L}(E)$, and $T$ is $A$-ergodic. 
\end{proof}

If the space is barrelled we have the Uniform Boundedness Principle at our disposal, and every $ A$-ergodic operator is $ A$-bounded. On the other hand, if the set $\{(AT)_nx:n\in \mathbb{N}_0\}$ is $\sigma(E,E^\prime)$-relatively compact, then it is bounded. If this happens for every $ x \in E$, again by the UBP, $AT$ is bounded in $ E$ and the operator $ T$ is $ A$-bounded. This altogether gives the following.

\begin{prop}\label{barrelled AE}
	Let $E$ be barrelled, $A$ a probability matrix, and $ T$ an operator satisfying \eqref{limit T invariant} for every $x\in E$. Then $T$ is $A$-ergodic if and only if $\{(AT)_nx:n\in \mathbb{N}_0\}$ is $\sigma(E,E^\prime)$-relatively compact for each $x\in E$.
\end{prop}

As we have just seen, the limit in \eqref{limit T invariant} plays an important role to describe when an $ A$-bounded operator is $ A$-ergodic. On the other hand  if $T$ is an $A$-ergodic operator then we clearly have 
\begin{equation}\label{restes AT}
	\lim_{n\to \infty}(AT)_{n+1}x-(AT)_nx=0 \,,
\end{equation}
for every $ x \in E$. Let us point out that, when $ A$ is the Cesàro matrix \eqref{sagrario}, these two conditions are equivalent. However, this is not always the case, take for instances  the operator $-\id:E\to E$ and the matrix 
\begin{equation*}
	A=
	\begin{pmatrix}
		1 & 0&0 & 0&0&\ldots \\
		1 &  0 & 0  &0& 0&\ldots \\
		\nicefrac{1}{2} & 0 &  \nicefrac{1}{2} &0&0& \ldots \\
		\nicefrac{1}{2} & 0 & \nicefrac{1}{2} &0&0& \ldots \\
		\nicefrac{1}{3} & 0 & \nicefrac{1}{3}  & 0 &\nicefrac{1}{3}  & \ldots\\
		\vdots & \vdots &\vdots & \vdots & \ddots
	\end{pmatrix}
\end{equation*}
Clearly $-\id$ is $A$-ergodic and satisfies \eqref{restes AT} for every $x\in E$ but it does not satisfy \eqref{limit T invariant} for any $x\neq 0$.

We rewrite now these two limits in \eqref{limit T invariant} and \eqref{restes AT} using matrices in order to obtain clean and manageable characterisations of this conditions. We define the matrix $\Delta:=(\delta_{n,i}-\delta_{n-1,i})$. That is
\begin{equation*}
	\Delta= 
	\begin{pmatrix}
		1 & 0&0 & \ldots \\
		-1 &  1 & 0  & \ldots \\
		0 & -1 &  1   & \ldots \\
		\vdots & \vdots & \vdots & \ddots
	\end{pmatrix}
\end{equation*}
Note that, for any sequence of operators $ \mathcal{T} $, we have
$ (\Delta \mathcal{T})_{n} = T_{n} - T_{n-1}$ for $ n\geq 1$ (and $ (\Delta \mathcal{T})_{0} = T_{0}$). Also, it is easy to check that this matrix is invertible, and its inverse is given by
\begin{equation*}
	\Delta^{-1}= 
	\begin{pmatrix}
		1 & 0&0 & \ldots \\
		1 &  1 & 0  & \ldots \\
		1 & 1 &  1   & \ldots \\
		\vdots & \vdots & \vdots & \ddots
	\end{pmatrix}
\end{equation*}

Let us see now how can we rewrite the limits in \eqref{limit T invariant} and \eqref{restes AT} by means of these matrices. 

\begin{nota}\label{Deltas}
	Observe in first place that, if $ T$ is any operator and $ A$ is a probability matrix, then (writing $ \mathcal{T} = \{(A T)_{n}  \}_{n}$)
	\begin{equation*}
		(\Delta A T)_{n+1} = (\Delta \mathcal{T})_{n+1} = (AT)_{n+1} - (AT)_n \,,
	\end{equation*}
	and $ T$ satisfies \eqref{restes AT} if and only if it is $\Delta A$-null.
\end{nota}

\begin{nota}\label{Deltas2}
	To handle the limit in \eqref{limit T invariant} we need to introduce an extra (mild) condition on the matrix, namely that the first column tends to $ 0$. Suppose then that $\lim_n a_{n0}=0$ and that a given operator $ T$ is $ A\Delta$-null. Then
	\begin{equation*}
		(A \Delta  T)_{n}= \big( A (\Delta T) \big)_{n}
		=a_{n0} \id +\sum_{i=1}^{n} a_{ni}(T^{i}-T^{i-1}) \,,
	\end{equation*}
	for every $ n \geq 1$. This gives the following identity
	\begin{equation*}
		T(A \Delta T)_{n} - a_{n0}\id= \sum_{i=0}^{n} a_{ni}(T^{i+1}-T^{i}) =(T-\id)(AT)_n\,.
	\end{equation*}
	Since $ T$ is continuous, this shows that  \eqref{limit T invariant} holds for every $ x \in E$.
\end{nota}

\begin{nota}
	If $ A$ is such that the pair $(\Delta A,A\Delta)$ satisfies \acc\!, then Proposition~\ref{schubert}--\ref{null to null} gives that every $ \Delta A$-null operator is $ A\Delta$-null. This means that the conditions \eqref{limit T invariant} and \eqref{restes AT} are, in some sense, equivalent (as it happens with the Cesàro matrix).
\end{nota}

Combining all the previous considerations we have the following.

\begin{teo}\label{coro Aergodic}
	Let $E$ be barrelled and $A$ a probability matrix such that $\lim_n a_{n0}=0$. Consider the following statements.
	\begin{enumerate}[(i)]
		\item \label{coro Aergodic1} $T$ is $A$-bounded and $A\Delta$-null, and the set $\{(AT)_nx:n\in \mathbb{N}_0\}$ is $\sigma(E,E^\prime)$-relatively compact.
		\item \label{coro Aergodic2} $T$ is $A$-ergodic.
		\item \label{coro Aergodic3} $T$ is $A$-bounded and $\Delta A$-null.
	\end{enumerate}
	Then~\ref{coro Aergodic1} implies~\ref{coro Aergodic2} and~\ref{coro Aergodic2} implies~\ref{coro Aergodic3}. Moreover, if the pair $(\Delta A,A\Delta)$ satisfies \acc\!, then \ref{coro Aergodic1} and~\ref{coro Aergodic2} are equivalent.
\end{teo}
\begin{proof}
	To see that \ref{coro Aergodic1} implies~\ref{coro Aergodic2} we have that the limit in \eqref{limit T invariant} is satisfied for every $x\in E$ using Remark~\ref{Deltas2}. Now Corollary~\ref{AE coro} gives that $T$ is $A$-ergodic. 
	
	Now for \ref{coro Aergodic2} implies~\ref{coro Aergodic3} we have that $E$ is barrelled then $T$ is $A$-bounded. Remark~\ref{Deltas} gives that $T$ is $\Delta A$-null and \ref{coro Aergodic3} follows. If additionally the pair $(\Delta A,A\Delta)$ satisfies \acc\!, since $T$ is $\Delta A$-null Proposition~\ref{schubert} gives that it is also $A\Delta$-null. Again using Corollary~\ref{AE coro} we obtain \ref{coro Aergodic1}.
\end{proof}

Let us recall that, if the space $ E$ is semi-reflexive, then every bounded set is $\sigma(E,E^\prime)$-relatively compact. Hence, if $ T$ is an $ A$-bounded operator on a reflexive space, then the set $\{(AT)_nx:n\in \mathbb{N}_0\}$ is $\sigma(E,E^\prime)$-relatively compact. A barrelled, semi-reflexive locally convex space is called reflexive. This observation together with Theorem~\ref{coro Aergodic} give the following result.  

\begin{teo} \label{pini roma}
	Let $A$ be a probability matrix such that $\lim_n a_{n0}=0$ and that $(\Delta A,A\Delta)$ satisfies \acc\!. If $ E$ is reflexive, then an operator $ T$ on $ E$ is $A$-ergodic if and only if $T$ is $A$-bounded and $A\Delta$-null.
\end{teo}

We can now relate ergodicity for different matrices under certain conditions. As a consequence of Proposition~\ref{schubert} and Theorem~\ref{coro Aergodic} we have the following.

\begin{teo} \label{suppe}
	Let  $A, B$ be two probability matrices such that $\lim_n b_{n0}=0$ holds. Assume the pair $(A,B)$ satisfies \ac\! and the pair $(\Delta A,B\Delta)$ satisfies \acc\!. If $ E$ is reflexive, then every $A$-ergodic operator on $ E$ is $B$-ergodic.
\end{teo}

In \cite[Proposition~3.3]{BoPaRi11} (see also \cite[Corollary 5.7]{llibre}) it is shown that every power bounded operator on a semi-reflexive space is mean ergodic. These are versions of the classical Yosida's Ergodic Theorem that when reformulated in terms of matrices (recall Remarks~\ref{identitat} and~\ref{cesaro}), this means that every $ I$-bounded operator is $ M$-ergodic.  With the same spirit we give now conditions on a pair of matrices that allow to pass from boundedness to ergodicity.
\begin{teo}\label{reflexive A Bergodic}
	Let  $A, B$ be two probability matrices such that $\lim_n b_{n0}=0$ holds. Assume the pair $(A,B)$ satisfies \ac\! and the pair $(A,B\Delta)$ satisfies \accc\!. If $ E$ is reflexive, then every $A$-bounded operator on $ E$ is $B$-ergodic.
\end{teo}

The previous results depend on conditions on $ \Delta A$ and/or $ B \Delta$, which in general may be difficult to check. We aim now at getting conditions that involve only the pair $ (A,B)$. We begin with a lemma that relates the limit in \eqref{limit T invariant} for two matrices.

\begin{lema} \label{albeniz}
	Let $ (A,B)$ be a pair of complex matrices satisfying \acc\!. Let $ T$ be an operator such that 
	\begin{equation*}
		\lim_{n} (T-\id) (AT)_{n} x = \lim_{n}  (AT)_{n} (T-\id) x = 0 
	\end{equation*}
	for every $ x \in E$. Then
	\begin{equation*}
		\lim_{n} (T-\id) (BT)_{n} x = \lim_{n}  (BT)_{n} (T-\id) x = 0 
	\end{equation*}
	for every $ x \in E$. 
\end{lema}
\begin{proof}
	First of all let us denote by $ K_{C}$ the supremum in \eqref{astc}, and observe that
	\begin{equation} \label{pollini}
		(T-\id) (AT)_{n} x = (T - \id) \Big(\sum_{i=0}^{n} a_{ni}T^{i} x\Big)
		= \sum_{i=0}^{n} a_{ni} \big(T^{i+1} x - T^{i}x\big) \,.
	\end{equation}
	Fix some $ \varepsilon >0$. Given a seminorm $ p$ and $ x \in X$ we can find $ n_{0} \in \mathbb{N}_0$ so that
	\begin{equation*}
		p \Big(\sum_{i=0}^{n} a_{ni} \big(T^{i+1} x - T^{i}x\big)\Big) < \frac{\varepsilon}{K_{C}}
	\end{equation*}
	for every $ n \geq n_{0}$. Now, proceeding as in \eqref{pollini} and taking into account that $ B=CA$, and that $ a_{ji}=0$ whenever $ i >j$,  we have
	\begin{multline*}
		(T-\id) (BT)_{n} x 
		= \sum_{i=0}^{n} \sum_{j=0}^{n} c_{nj} a_{ji} \big(T^{i+1} x - T^{i}x\big) \\
		=  \sum_{j=0}^{n} c_{nj}  \sum_{i=0}^{n}a_{ji} \big(T^{i+1} x - T^{i}x\big) 
		=  \sum_{j=0}^{n} c_{nj}  \sum_{i=0}^{j}a_{ji} \big(T^{i+1} x - T^{i}x\big) \,.
	\end{multline*}
	If $ n \geq n_{0}$ then
	\begin{align*}
		p \big(& (T-\id) (BT)_{n} x \big) \\
		& \leq p\Big(\sum_{j=0}^{n_{0}} c_{nj}  \sum_{i=0}^{j}a_{ji} \big(T^{i+1} x - T^{i}x\big)\Big)
		+p\Big(\sum_{j=n_{0}+1}^{n} c_{nj}  \sum_{i=0}^{j}a_{ji} \big(T^{i+1} x - T^{i}x\big)\Big) \\
		& \leq \max_{1 \leq j \leq n_{0}} p\big( (T-\id) (AT)_{j} x\big)\sum_{j=0}^{n_{0}} \vert c_{nj} \vert
		+ \frac{\varepsilon}{K_{C}}\sum_{j=n_{0}+1}^{n} \vert c_{nj} \vert
		\leq K_{2} \sum_{j=0}^{n_{0}} \vert c_{nj} \vert + \varepsilon \,.
	\end{align*}
	Letting $ n \to \infty$ the first summand tends to $ 0$, so that
	\begin{equation*}
		q \big( (T-\id) (BT)_{n} x \big) < \varepsilon\,.
	\end{equation*}
	Since $ \varepsilon >0$ was arbitrary, this proves the claim.
\end{proof}

\begin{teo}\label{encina}
	Let  $A, B$  be two probability matrices such that $(A,B)$ satisfies \break\acc\!. If $ E$ is reflexive, then every  $A$-bounded operator satisfying \eqref{limit T invariant} (for the matrix $ A$ for every $ x \in E$) is $ B$-ergodic.
\end{teo}
\begin{proof}
	To begin with, since $ T$ is $ A$-bounded and $ (A,B)$ satisfies \ac\!, by Proposition~\ref{schubert}, it is also $ B$ -bounded. Then the set $\{(BT)_nx:n\in \mathbb{N}_0\}$ is bounded for every $x\in E$ and (being $ E$ reflexive) $\sigma(E,E^\prime)$-relatively compact. Finally, Lemma~\ref{albeniz} shows that \eqref{limit T invariant} holds for the matrix $ B$ for every $ x \in E$, and Corollary~\ref{AE coro} gives the conclusion.
\end{proof}

\begin{nota}
	Suppose $ E$ is a Montel space, and $ T$ is an $ A$-ergodic operator. Then, the sequence  $\{(A T)_n:n\in \mathbb{N}_0\}$  converges pointwise to some operator $ P$. On the other hand, $\{(AT)_n:n\in \mathbb{N}_0\}$ is equicontinuous and, by  \cite[\S39-4-(2), p.~139]{K2}, the topology of pointwise convergence and of uniform convergence on precompact sets coincide on this set. Finally, since $E$ is Montel, bounded sets and relatively compact sets coincide. This altogether shows that  $(AT)_n$ converges to $P$ uniformly on bounded sets of $E$, and $T$ is uniformly $A$-ergodic.
\end{nota}

This gives the following analogues of Theorems~\ref{pini roma} and~\ref{reflexive A Bergodic} (see \cite[p.~917]{BoPaRi11} and \cite[Proposition~5.11]{llibre})

\begin{prop}
	Let $A$ be a probability matrix such that $\lim_n a_{n0}=0$ and that $(\Delta A,A\Delta)$ satisfies \acc\!. If $ E$ is Montel, then an operator $ T$ on $ E$ is uniformly $A$-ergodic if and only if $T$ is $A$-bounded and $A\Delta$-null.
\end{prop}

\begin{prop} \label{semiMontel UAE}
	Let  $A, B$ be two probability matrices such that $\lim_n b_{n0}=0$ holds. Assume the pair $(A,B)$ satisfies \ac\! and the pair $(A,B\Delta)$ satisfies \accc\!. If $ E$ is Montel, then every $A$-bounded operator on $ E$ is uniformly $B$-ergodic.
\end{prop}

\begin{nota}
	In Theorems~\ref{reflexive A Bergodic} and~\ref{encina}, one can replace the condition that $E$ is reflexive with the weaker assumption that $E$ is semi-reflexive. This is possible since the operators are assumed to be $A$-bounded and the barrelledness of $E$ is no longer required. In the same spirit ``Montel'' can be replaced by ``semi-Montel'' in Proposition~\ref{semiMontel UAE}. We omit these additional results to avoid redundancy.
\end{nota}

\section{Variants of classical notions}

Our aim now is to relate some classical notions (power and Cesàro boundedness, and  mean ergodicity) with other ones defined by other matrices. Recall the power boundedness is defined by the $ \id$, and Cesàro boundedness and mean ergodicity by the matrix $ M$ from \eqref{sagrario}. \\
For each fixed $ p \in \mathbb{R}$ we define $ S(n,p) = \sum_{i=1}^{n} i^{p}$ for $n \in \mathbb{N}$ and consider the matrix $M_{p} = \Big(\frac{i^{p}}{S(n,p)}\Big)_{1 \leq i \leq n}$, where the elements with $ i>n$ are $ 0$. That is
\begin{equation}\label{marais}
M_{p}= 
\begin{pmatrix}
1 & 0&0 &0 & \ldots \\
\frac{1}{S(2,p)} &  \frac{2^{p}}{S(2,p)}  & 0 & 0 & \ldots \\
\frac{1}{S(3,p)} & \frac{2^{p}}{S(3,p)} &  \frac{3^{p}}{S(3,p)}  &0 & \ldots \\
\frac{1}{S(4,p)} & \frac{2^{p}}{S(4,p)} &  \frac{3^{p}}{S(4,p)}  &  \frac{4^{p}}{S(4,p)}  & \ldots \\
\vdots & \vdots & \vdots & \vdots & \ddots
\end{pmatrix}
\end{equation}
For each $ n \in \mathbb{N}$, the sum of the $ n$-th row is $  \sum_{i=1}^{n}\frac{i}{S(n,i)}=1 $;  then $ M_{p}$ is a probability matrix. Note that $ M_{0}$ is exactly \eqref{sagrario}.\\
In a similar fashion we define $S(n,\exp)=\sum_{i=1}^{n}e^i$ for $n\in\mathbb{N}$ and consider the matrix $M_{\exp} =\Big(\frac{e^i}{S(n,\exp)}\Big)_{1 \leq i \leq n}$, where the elements with $ i>n$ are $ 0$. That is
\begin{equation*}
	M_{\exp}= 
	\begin{pmatrix}
		1 & 0&0 &0 & \ldots \\
		\frac{e}{S(2,\exp)} &  \frac{e^2}{S(2,\exp)}  & 0 & 0 & \ldots \\
		\frac{e}{S(3,\exp)} & \frac{e^2}{S(3,\exp)} &  \frac{e^3}{S(3,\exp)}  &0 & \ldots \\
		\frac{e}{S(4,\exp)} & \frac{e^2}{S(4,\exp)} &  \frac{e^3}{S(4,\exp)}  &  \frac{e^4}{S(4,\exp)}  & \ldots \\
		\vdots & \vdots & \vdots & \vdots & \ddots
	\end{pmatrix}
\end{equation*}
Again $ M_{\exp}$ is a probability matrix.\\
Given an operator $ T$ and $ n \in \mathbb{N}$ we have (see \eqref{panadera})
\begin{equation*}
(M_{p}T)_{n} = \frac{1}{S(n,p)} \sum_{k=1}^{n} k^{p} T^{k-1} \,,
\end{equation*}
and 
\begin{equation*}
	(M_{\exp}T)_{n} = \frac{1}{S(n,\exp)} \sum_{k=1}^{n} e^k T^{k-1} \,.
\end{equation*}
What we are going to do now is to relate the ergodic properties defined by these matrices with the classical ones of power boundedness and Cesàro boundedness/ergodicity.  Our strategy is, then to relate (in terms of \ac and \acc\!) these matrices with $\id$ and $ M$ from \eqref{sagrario}. Before we proceed, let us establish some basic properties of these matrices.

\begin{nota} \label{aretino}
For each $ p \in \mathbb{R} $ the matrix $ M_{p}$ is invertible, and its inverse is the matrix $ D_{p}= \big(d_{ij} \big)$, where 
\begin{itemize}
\item $ d_{nn} = \frac{S(n,p)}{n^{p}}$ for $ n \geq 1$, 
\item $ d_{n \, (n-1)} = - \frac{S(n-1,p)}{n^{p}}$ for $ n \geq 2$, 
\item $ 0$ otherwise
\end{itemize}
That is, 
\begin{equation*}
M_{p}^{-1} = D_{p} = 
\begin{pmatrix}
\frac{S(1,p)}{1^{p}} & 0&0 &0 & \ldots \\
- \frac{S(1,p)}{2^{p}} & \frac{S(2,p)}{2^{p}}&0 &0 & \ldots \\
0 & - \frac{S(2,p)}{3^{p}} & \frac{S(3,p)}{3^{p}}&0 & \ldots \\
0&0 & - \frac{S(3,p)}{4^{p}} & \frac{S(4,p)}{4^{p}}& \ldots \\
\vdots & \vdots & \vdots & \vdots & \ddots
\end{pmatrix}
\end{equation*}
This can be seen by direct computation. If we write $ M_{p}=(m_{ij})$ , then the $ (n\, i)$ element of the product $ D_{p} M_{p}$ is given by
\begin{equation} \label{joan}
\sum_{j=1}^{n} d_{nj} m_{ji}= - \frac{S(n-1,p)}{n^{p}} m_{(n-1) \, i} +  \frac{S(n,p)}{n^{p}} m_{n i} 
\end{equation} 
Now, if $ i>n$, then $ m_{(n-1) \, i} = m_{n i}  =0$, and the sum is $ 0$. If $ i <n$, then 
\begin{multline*}
- \frac{S(n-1,p)}{n^{p}} m_{(n-1) \, i} +  \frac{S(n,p)}{n^{p}} m_{n i}\\
= - \frac{S(n-1,p)}{n^{p}} \frac{i^{p}}{S(n-1,p)}  +  \frac{S(n,p)}{n^{p}} \frac{i^{p}}{S(n,p)} =0\,.
\end{multline*} 
Finally, for $ i=n$ we have $ m_{(n-1)\, n}=0$ and $ m_{nn} =  \frac{n^{p}} {S(n,p)} $, so that the sum in \eqref{joan} is $ 1$.\\
Similar computations show that
\begin{equation*}
	M_{\exp}^{-1} = 
	\begin{pmatrix}
		\frac{S(1,\exp)}{e} & 0&0 &0 & \ldots \\
		- \frac{S(1,\exp)}{e^2} & \frac{S(2,\exp)}{e^2}&0 &0 & \ldots \\
		0 & - \frac{S(2,\exp)}{e^3} & \frac{S(3,\exp)}{e^3}&0 & \ldots \\
		0&0 & - \frac{S(3,\exp)}{e^4} & \frac{S(4,\exp)}{e^4}& \ldots \\
		\vdots & \vdots & \vdots & \vdots & \ddots
	\end{pmatrix}
\end{equation*}

\end{nota}

\begin{nota} \label{heiland}
We also need to get some bounds for $ S(n,p)$. Standard computations show that 
\begin{equation*}
S(n,p) \sim \begin{cases}
1 & \text{ if } p<-1 \\ \log n &\text{ if } p=-1\\ 
n^{p+1} &\text{ if } p> -1
\end{cases}
\end{equation*}
\end{nota}

\subsection{Power boundedness}

Our next step is to see when the pairs formed by these matrices and the identity satisfy \ac or \acc\!.

\begin{prop} \label{bartleby}
\phantom{text}
\begin{enumerate}[(i)]
\item\label{bartlebya} The pair $ (\id , M_{p})$ satisfies \ac for every $ p \in \mathbb{R}$\,.

\item\label{bartlebyb}  The pair $ ( M_{p}, \id)$ does not satisfy \ac for any $ p \in \mathbb{R}$\,.

\item\label{vendettaa} The pair $ (\id , M_{\exp})$ satisfies \acc\!.

\item\label{vendettab}  The pair $ ( M_{\exp}, \id)$ satisfies \acc\!.
\end{enumerate}
\end{prop}
\begin{proof}
Since $ M_{p} \id = M_{p}$  and each row sums up to $ 1$,  \eqref{astc} is clearly satisfied and we have \ref{bartlebya}.\\
Fix now $ p \in \mathbb{R}$ and note that, since $ M_{p}$ is invertible, in order to see \ref{bartlebyb}, we need to check that $ M_{p}^{-1}$ does not satisfy  \eqref{astc}. The sum appearing there is in this case (see Remark~\ref{aretino})
\begin{equation*}
\frac{S(n,p)}{n^{p}} + \frac{S(n-1,p)}{n^{p}} 
= 1+  \frac{2S(n-1,p)}{n^{p}} \,.
\end{equation*}
Now, Remark~\ref{heiland} gives us what we need to conclude the argument.
\begin{equation*}
\frac{2S(n-1,p)}{n^{p}} \gg \begin{cases}
\dfrac{1}{n^{p}} & \text{ if }p < -1 \\ & \\
n \log(n-1)& \text{ if }p = -1 \\ & \\
 \Big(1- \dfrac{1}{n}\Big)^{p} (n-1) & \text{ if }p > -1
\end{cases}
\end{equation*}
In each case, the lower bound tends to  $ \infty $ as $ n \to \infty$. This gives the conclusion.\\
Since $M_{\exp}\id= M_{\exp}$ and each row sums up to $1$, \eqref{astc} is satisfied. We have that $\frac{e^k}{S(n,\exp)}\to 0$ as $n\to\infty$, for every $k\in\mathbb{N}$, thus we obtain \eqref{ast2c} and \ref{vendettaa} follows.\\
Finally to see \ref{vendettab}, we only need to check that $M^{-1}_{\exp}$ satisfies \eqref{astc} since \eqref{ast2c} is clear. We have that 
	\begin{equation*}
		S(n,\exp)<\int_{1}^{n+1} e^x\dx =e^{n+1}-e\,,
	\end{equation*}
	for every $n\in\mathbb{N}$. Then we can bound the sum for the matrix $M^{-1}_{\exp}$ as follows
	\begin{align*}
		\frac{S(n,\exp)}{e^{n+1}}  + \frac{S(n+1,\exp)}{e^{n+1}} < \frac{e^{n+1}-e}{e^{n+1}} + \frac{e^{n+2}-e}{e^{n+1}}< 1 + e\,,
	\end{align*}
	for every $n\in\mathbb{N}$, and \ref{vendettab} holds.
\end{proof}

As a straightforward consequence of this and Theorem~\ref{cohen} we have

\begin{prop}
Let $ T$ be an operator on a  lcHs, then
\begin{enumerate}[(i)]
\item if $ T$ is  power bounded, then it is $ M_{p}$-bounded for every $ p \in \mathbb{R}$.

\item  $ T$ is  power bounded if and only if it is $M_{\exp}$-bounded.
\end{enumerate}
\end{prop}

We can also describe when the powers of an operator converge pointwise in terms of $ M_{\exp}$.

\begin{prop}\label{cuidate}
	Let $ X$ be a reflexive lcHs and $ T: X \to X$  a continuous, linear operator. Then the following statements are equivalent.
	\begin{enumerate}[(i)]
		\item \label{cuidate1} The sequence of iterates $(T^n)_n$ converges pointwise to some operators in $X$.
		\item \label{cuidate2} $ T$  is $\id$-ergodic.
		\item \label{cuidate3} $ T$  is $ M_{\exp}$-ergodic.
	\end{enumerate}
\end{prop}
\begin{proof}
The equivalence between \ref{cuidate1} and \ref{cuidate2} is trivial. Let us see that \ref{cuidate2} implies \ref{cuidate3}. Since $T$ is $\id$-ergodic we have that $\lim_{n\to\infty}T^{n+1}x - T^n x=0$, for every $x\in X$, and \eqref{limit T invariant} is satisfied for the matrix $\id$ for every $x\in X$. By Proposition~\ref{bartleby} and Theorem~\ref{encina}, $T$ is $M_{\exp}$-ergodic.\\
Using the same results, to see that \ref{cuidate3} implies \ref{cuidate2} it is enough to show \eqref{limit T invariant} for the matrix $M_{\exp}$ for every $x\in X$. Since $T$ is $M_{\exp}$-ergodic we have $\lim_n (M_{\exp}T)_{n+1}x- (M_{\exp}T)_{n}x=0\,,$
for every $x\in X$. Observe that 
\[ 
\lim_{n\to\infty}\frac{S(n+1,\exp)}{e \cdot S(n,\exp)}=\lim_{n\to\infty}\frac{e+e^2+\dots e^{n+1}}{e^2+\dots + e^{n+1}}= 1\,. 
\]
Then we have
	\begin{align*}
		0=& \lim_{n\to\infty} \frac{S(n+1,\exp)}{e \cdot S(n,\exp)} (M_{\exp}T)_{n+1}x - (M_{\exp}T)_{n}x \\
		=& \lim_{n\to\infty} \frac{1}{e \cdot S(n,\exp)}  \sum_{k=1}^{n+1} e^k T^{k-1}x - \frac{1}{S(n,\exp)}  \sum_{k=1}^{n} e^k T^{k-1}x \\
		=& \lim_{n\to\infty} \frac{1}{S(n,\exp)}  \sum_{k=0}^{n} e^k T^{k}x - \frac{1}{S(n,\exp)}  \sum_{k=1}^{n} e^k T^{k-1}x \\
		=& \lim_{n\to\infty} \frac{Tx}{S(n,\exp)} + \frac{1}{S(n,\exp)}  \sum_{k=1}^{n} e^k (T-\id)T^{k-1}x \\
		=& \lim_{n\to\infty} (T-\id)(M_{\exp}T)_n x\,
	\end{align*}
for every $x\in X$. This yields the conclusion.
\end{proof}

\subsection{Cesàro boundedness}

We now relate Cesàro boundedness with $ M_{p}$-boundedness. Keeping in mind that $ M_{0}$ is exactly the matrix in \eqref{sagrario}, we need to know when the pairs $ (M_{p}, M_{0})$ and $ (M_{0}, M_{p})$ have property \ac\!.

\begin{prop} \label{musgania}
Let $ p,q \in \mathbb{R}$ satisfy one of the following two conditions
\begin{equation*}
q \leq p \, \text{ or } \,  -1 < p < q \,.
\end{equation*}
Then the pair $ (M_{p},M_{q})$ satisfies \ac\!.
\end{prop}
\begin{proof}
We have to see that in each case the matrix $ C=M_{q}M_{p}^{-1}$ satisfies \eqref{astc}. If $ p=q$, this is the identity matrix and it is trivially satisfied. To handle the case $ p \neq q$ we see  what do the elements $ c_{ni}$ of $ C$ look like. We write $ M_{q}=m_{ni}^{q}$ and $ M_{p}^{-1}= d_{ni}^{p}$. Then 
\begin{equation*}
c_{nn} = m_{nn }^{q}d_{nn}^{p} = \frac{n^{q}}{S(n,q)} \cdot \frac{S(n,p)}{n^{p}} = \frac{S(n,p)}{S(n,q)}n^{q-p} \,.
\end{equation*}
If $ i\geq n+1$ then clearly $ c_{ni}=0$, and for $ i \leq n-1$ we have
\begin{equation} \label{adoracion}
\begin{split}
c_{ni} & = \sum_{j=1}^{n}m_{nj}^{q} d_{ji}^{q} = m_{ni}^{q} d_{ii}^{q} + m_{n(i+1)}^{q} d_{(i+1)i}^{q} \\
& = \frac{i^{q}}{S(n,q)} \cdot \frac{S(i,p)}{i^{p}} + \frac{(i+1)^{q}}{S(n,q)} \cdot \frac{-S(i,p)}{(i+1)^{p}} \\
&= \frac{S(i,p)}{S(n,q)} \big(i^{q-p} - (i+1)^{q-p}  \big) \,.
\end{split}
\end{equation}
Suppose now that $ q<p$. Then $ c_{ni} \geq 0$ for every $ n, i$ and, for each fixed $ n \in \mathbb{N}$ we have
\begin{align*}
\sum_{i=1}^{\infty} \vert c_{ni} \vert
& = \sum_{i=1}^{n} c_{ni} \\
& = \frac{1}{S(n,q)} \bigg( \sum_{i=1}^{n-1} S(i,p) \big(i^{q-p} - (i+1)^{q-p}\big) + S(n,p) n^{q-p} \bigg) \\
& = \frac{1}{S(n,q)} \bigg( \sum_{i=1}^{n-1} S(i,p) i^{q-p} - \sum_{i=2}^{n} S(i-1,p) i^{q-p}\big) + S(n,p) n^{q-p} \bigg) \\
& = \frac{1}{S(n,q)} \bigg( \sum_{i=2}^{n} i^{p} i^{q-p} +1  \bigg)
= \frac{S(n,q)}{S(n,q)} = 1\,.
\end{align*}
This completes the proof when $ q \leq p$.\\
Let us suppose now $p <q$ are arbitrary and fix $ n \in \mathbb{N}$; then
\begin{align*}
\sum_{i=1}^{\infty} \vert c_{ni} \vert
& = \frac{1}{S(n,q)} \bigg(\sum_{i=1}^{n-1} \frac{S(i,p)}{S(n,q)} \big( (i+1)^{q-p} - i^{q-p}\big) + \frac{S(n,p)}{S(n,q)} n^{q-p}  \bigg)\\
& =\frac{1}{S(n,q)} \bigg( \sum_{i=2}^{n} S(i-1,p) i^{q-p} - \sum_{i=1}^{n-1} S(i,p)i^{q-p} + S(n,p) n^{q-p} \bigg)\\
& =\frac{1}{S(n,q)} \bigg( S(n-1,p) n^{q-p}-1 + S(n,p) n^{q-p} - \sum_{i=2}^{n-1}  i^{p} i^{q-p} \bigg) \\
&= \frac{1}{S(n,q)} \bigg( 2 S(n-1,p) n^{q-p}+ n^{p}  n^{q-p} -  S(n-1,q) i^{p} i^{q-p} \bigg) \\
&= \frac{2 S(n,p) n^{q-p}-S(n,q)}{S(n,q)} = 2 \, n^{q-p}  \frac{S(n,p)}{S(n,q)} -1 \,.
\end{align*}
Assuming $ -1<p<q$, using Remark~\ref{heiland} immediately shows that \eqref{astc} holds and completes the proof.
\end{proof}

\begin{nota}
Let us note that $ -1$ is actually a sort of a border case, in the sense that if $ p \leq -1$ and $p<q$, then the pair $ (M_{p},M_{q})$ does not satisfy \break\ac\!. In this case Remark~\ref{heiland} gives
\begin{equation*}
\sum_{i=1}^{\infty} \vert c_{ni} \vert
=  2 \, n^{q-p}  \frac{S(n,p)}{S(n,q)} -1
\sim
\begin{cases}
n^{q-p} & \text{ if } p<q<-1 \\
\dfrac{n^{q-p} }{\log n}& \text{ if } p<q=-1 \\
n^{-(p+1)}& \text{ if } p<-1 <q \\
\log n & \text{ if } p=-1<q  
\end{cases}
\end{equation*}
In all cases the right-hand side tends to $ \infty$ as $ n \to \infty$\,.
\end{nota}

As a straightforward consequence of Propositions~\ref{cohen} and~\ref{musgania} we have the following.

\begin{teo} \label{nana}
Let $ X$ be a lcHs and $ T: X \to X$  a continuous, linear operator. Then the following statements are equivalent.
\begin{enumerate}[(i)]
\item $ T $ is Cesàro-bounded.
\item $ T$  is  $ M_{p}$-bounded for some $ p>-1$.
\item $ T$  is  $ M_{p}$-bounded for all $ p>-1$.
\end{enumerate}
\end{teo}

\subsection{Mean ergodicity}

To compare ergodicity for these matrices we have to see when the pair $ (M_{p},M_{q})$ satisfies \acc\!.
\begin{prop} \label{cordero}
Let $ p,q \in \mathbb{R}$ be such that $ -1 < p$ and $ -1 \leq q$. Then the pair $ (M_{p},M_{q})$ satisfies \acc\!.
\end{prop}
\begin{proof}
To begin with, by Proposition~\ref{musgania}, the pair satisfies \ac\!. To complete the proof we just have to check that the matrix $ C=M_{q}M_{p}^{-1}$ satisfies \eqref{ast2c}\!. For each fixed $ i_{0} \in \mathbb{N}$ we have (recall \eqref{adoracion})
\begin{equation*}
c_{ni_{0}} =  \frac{1}{S(n,q)} \Big(S(i_{0},p) \big(i_{0}^{q-p} - (i_{0}+1)^{q-p}  \big)\Big) \,.
\end{equation*}
Since $\lim_{n} S(n,q) = \sum_{k=1}^{\infty}  k^{q} = \infty$ and the second factor is constant, we have the result.
\end{proof}

Combining this and Theorem~\ref{encina} we have the following.

\begin{teo} \label{matermea}
Let $ X$ be a reflexive lcHs and $ T: X \to X$  a continuous, linear operator. Then the following statements are equivalent.
\begin{enumerate}[(i)]
\item \label{matermea1} $ T $ is mean ergodic.
\item \label{matermea2} $ T$  is  $ M_{p}$-ergodic for some $ p>-1$.
\item \label{matermea3} $ T$  is  $ M_{p}$-ergodic for all $ p>-1$.
\end{enumerate}
\end{teo}
\begin{proof}
Let us see that~\ref{matermea1} implies~\ref{matermea3}. Fix $p>-1$. Since $X$ is reflexive, $T$ is $M_0$-ergodic and \eqref{limit T invariant} is satisfied for the matrix $M_0$ for every $x\in X$. By Proposition~\ref{cordero} and Theorem~\ref{encina}, $T$ is $M_p$-ergodic.\\	
To see that~\ref{matermea2} implies~\ref{matermea1}, assume $p>-1$ is different from $0$. It is clear that $T$ is $M_p$-bounded and therefore $M_0$-bounded (Theorem~\ref{nana}). Observe that 
\[ 
\lim_{n\to\infty} (M_pT)_{n+1}x - (M_pT)_n x=0 
\]
is satisfied for every $x\in X$ (see \eqref{restes AT}). Now since $\lim_{n}\frac{S(n+1,p)}{S(n,p)}=1$, we have that 
\begin{multline*}
		 0=\lim_{n\to\infty} \frac{S(n+1,p)}{S(n,p)}(M_pT)_{n+1}x - (M_pT)_n x \\
		 =\lim_{n\to\infty} \sum_{i=1}^{n+1} \frac{S(n+1,p)}{S(n,p)} \frac{i^p}{S(n+1,p)} T^{i-1}x  -  \sum_{i=1}^{n} \frac{i^p}{S(n,p)} T^{i-1}x \\
		 = \lim_{n\to\infty} \frac{(n+1)^p}{S(n,p)}T^{n}x\,,
\end{multline*}
for every $x\in X$. Using Remark~\ref{heiland} gives that $\lim_n \frac{T^nx}{n}=0$ for every $x\in X$, which is equivalent to \eqref{limit T invariant} for the matrix $M_0$ for every $x\in X$. Corollary~\ref{AE coro} yields the conclusion. Since \ref{matermea3} clearly implies \ref{matermea2} the proof is completed
\end{proof}

\section{Passing from boundedness to ergodicity}

In the previous section we have related power/Cesàro boundedness and mean ergodicity with boundedness and ergodicity for other matrices. From a more general point of view, this can be seen as finding conditions on a pair $ (A,B)$ so that  every $ A$-bounded operator is $ B$-bounded, or every $ A$-ergodic is $ B$-ergodic.  Another interesting feature is to find pairs that allow to pass from boundedness to ergodicity; more precisely we are now interested in pairs $ (A,B)$ in such a way that $ A$-bounded operators  are $ B$-ergodic. We will say that the pair satisfies the (BE)-property if this holds for every operator on a reflexive space. Theorem~\ref{reflexive A Bergodic} gives conditions on a pair to satisfy this property, namely
\begin{itemize}
	\item $\lim_{n} b_{n0}=0$,
	\item $(A,B)$ satisfies \ac and
	\item  $(A,B\Delta)$ satisfies \accc\!.	
\end{itemize}

Since, as we already noted, our starting point is  power and Cesàro boundedness, and mean ergodicity, we are mostly interested in linking the matrices defining these properties ($ \id$ and $ M_{0}$, recall \eqref{marais} and \eqref{sagrario}) with other matrices. It is a well known fact (see e.g. \cite[Corollary~5.7]{llibre}) that the pair $ (\id, M_{0})$ satisfies the (BE)-property. This can also be now easily deduced from the fact that $M_0\Delta= \left(\frac{1}{n}\delta_{n,i}\right)_{1\leq i\leq n}$ and then, $ (\id, M_0\Delta)$ clearly satisfies \accc\!.	

\begin{prop}\label{id 0 -1}
Every Cesàro bounded operator on a reflexive space is $ M_{-1}$-ergodic.
\end{prop}
\begin{proof}
Taking \eqref{marais} and Proposition~\ref{musgania} into account, we just have to see that the pair $(M_0,M_{-1}\Delta)$ satisfies \accc\!. To do that note first that, if we denote $M_{-1}\Delta=(d_{ni})$, we have
\begin{equation*}
d_{ni}= \begin{cases}
\frac{1}{iS(n,-1)}-\frac{1}{(i+1)S(n,-1)} & \text{ if } i<n \\
\frac{1}{nS(n,-1)} & \text{ if } i=n \\
0  & \text{ if } i>n \\
\end{cases}
\end{equation*}
Then, writing  $M_0^{-1}=(m_{ni})$, the  product  $M_{-1}\Delta M_0^{-1}=(c_{ni})$ is given by
\begin{multline*}
c_{ni}  = \sum_{j=1}^{n}d_{nj} m_{ji} = \sum_{j=1}^{n-1} \frac{1}{S(n,-1)}\left(\frac{1}{j}-\frac{1}{j+1}\right)m_{ji} + \frac{m_{ni}}{nS(n,-1)}\\
 = \sum_{j=1}^{n-1} \frac{1}{S(n,-1)}\frac{1}{j(j+1)}m_{ji} + \frac{m_{ni}}{nS(n,-1)} \,.
\end{multline*}
Recall now that $m_{ni}=n\delta_{n,i}-(n-1)\delta_{n-1,i}$ and let us consider different cases
\begin{itemize}
	\item if $i\leq n-2$\, then
		\begin{multline*}
		c_{ni}  = \sum_{j=1}^{n-1} \frac{1}{S(n,-1)}\frac{1}{j(j+1)}m_{ji} = \frac{1}{S(n,-1)}\left( \frac{i}{i(i+1)} - \frac{i}{(i+1)(i+2)} \right) \\
		 = \frac{i}{(i+1)S(n,-1)}\left( \frac{1}{i} - \frac{1}{i+2}\right)= \frac{2}{(i+1)(i+2)S(n,-1)}\,,
		\end{multline*}
		
	\item if $i=n-1$, then 
	\begin{equation*}
	c_{n(n-1)}= \frac{1}{S(n,-1)}\frac{n-1}{(n-1)n}-\frac{n-1}{nS(n,-1)}=\frac{-n+2}{nS(n,-1)} \,,
	\end{equation*}
	
	\item if $i=n$, then $c_{nn}= \frac{n}{nS(n,-1)}=\frac{1}{S(n,-1)}$\,, and
	\item $c_{ni}=0$ for $i>n$.
\end{itemize}
Taking all this into account we finally have
\begin{multline*}
	\lim_{n\to \infty} \sum_{i=1}^\infty|c_{ni}| = \lim_{n\to \infty} \frac{1}{S(n,-1)} \left(\sum_{i=1}^{n-2} \frac{2}{(i+1)(i+2)} + \frac{n-2}{n} + 1\right) \\
	\leq \lim_{n\to \infty} \frac{1}{S(n,-1)} \left(\sum_{i=1}^{n-2} \frac{2}{i^2}  + 2\right)
	\leq \lim_{n\to \infty} \frac{1}{S(n,-1)} \left(2+ 2\frac{\pi^2}{6}\right)=0\,,
\end{multline*}
which gives the conclusion.
\end{proof}

We have, then, that the pairs $ (\id , M_{0})$ and $ (M_{0}, M_{-1})$ have the (BE)-property. We are now interested in finding a probability matrix (say $ M_{g}$) so that both $(\id,M_g)$ and $(M_g,M_0)$ have the (BE)-property, that is, in reflexive spaces every power bounded operator is $ M_{g}$-ergodic, and every $ M_{g}$-bounded is mean ergodic. To construct such a matrix we consider a continuous strictly positive function $f:[1,+\infty[\to \mathbb{R}^+$. Then we define $S(n,f):=\sum_{i=1}^n f(i)$, which is increasing in $n$, and  the matrix
\begin{equation} \label{machaut}
M_{f}:= 
\begin{pmatrix}
1 & 0&0 &0 & \ldots \\
\frac{f(1)}{S(2,f)} &  \frac{f(2)}{S(2,f)}  & 0 & 0 & \ldots \\
\frac{f(1)}{S(3,f)} & \frac{f(2)}{S(3,f)} &  \frac{f(3)}{S(3,f)}  &0 & \ldots \\
\frac{f(1)}{S(4,f)} & \frac{f(2)}{S(4,f)} &  \frac{f(3)}{S(4,f)}  &  \frac{f(4)}{S(4,f)}  & \ldots \\
\vdots & \vdots & \vdots & \vdots & \ddots
\end{pmatrix}
\end{equation}
This is a probability matrix, and a simple computation shows that its inverse is given by 
\begin{equation} \label{12473BF}
M_{f}^{-1} = 
\begin{pmatrix}
\frac{S(1,f)}{f(1)} & 0&0 &0 & \ldots \\
- \frac{S(1,f)}{f(2)} & \frac{S(2,f)}{f(2)}&0 &0 & \ldots \\
0 & - \frac{S(2,f)}{f(3)} & \frac{S(3,f)}{f(3)}&0 & \ldots \\
0&0 & - \frac{S(3,f)}{f(4)} & \frac{S(4,f)}{f(4)}& \ldots \\
\vdots & \vdots & \vdots & \vdots & \ddots
\end{pmatrix}
\end{equation}
If the pair $(M_f,M_0\Delta)$ satisfies \accc (this is one of the conditions giving that  the pair $(M_f,M_0)$ satisfies the (BE)-property), then the sums of the rows (in absolute value) of the matrix $ M_0\Delta M_f^{-1}$ have to tend to $ 0$ (recall \eqref{ast3c}), that is we need 
\begin{equation}\label{lim sf}
\lim_{n\to \infty} \frac{1}{n-1}\frac{S(n-1,f)}{f(n)} + \frac{1}{n}\frac{S(n,f)}{f(n)}=0 \,.
\end{equation}
A way to achieve this is by adjusting the  function $f$ in such a way that the element $\frac{S(n,f)}{f(n)}$ grows as $\sqrt{n}$ (other rates of growth can be considered). This is satisfied by the solutions of the integral equation
\begin{equation*}
 \int_{1}^{y}f(x) \dx = \sqrt{y}f(y) + K\,,
\end{equation*}
which are of the form $f(x)=C\frac{e^{2\sqrt{x}}}{\sqrt{x}}$.
These are natural candidates to be the function defining our matrix.

\begin{nota}
Take  the function  $f(x)=\frac{e^{2\sqrt{x}}}{\sqrt{x}}$, consider  the matrix $ M_{f}$ defined as in \eqref{machaut} and let us see that both $(\id, M_f)$ and $(M_f, M_0)$ have the (BE)-property. \\
To handle the pair $(\id, M_f)$ we note first that  $\lim_{n} \frac{f(1)}{S(n,f)}=0$, and Condition \ac is automatic since $M_f$ is a probability matrix. It is left to see that $(\id,M_{f}\Delta)$ satisfies \accc\!. Since $f$ is increasing we have that in this case \eqref{ast3c} is satisfied if and only if 
\[ 
\lim_{n\to \infty} \frac{f(n)}{S(n,f)}=0\,. 
\]
Observe that, for $n>1$, the following inequalities hold
\begin{equation}\label{ineqg}
 S(n-1,f)\leq \int_1^n f(x)\dx \leq S(n,f)\,, 
\end{equation}
therefore 
\[ 
\lim_{n\to \infty} \frac{f(n)}{S(n,f))}
\leq \lim_{n\to \infty} \frac{f(n)}{\int_1^n f(x)\dx}
= \lim_{n\to \infty} \frac{e^{2\sqrt{n}}}{\sqrt{n} (e^{2\sqrt{n}}-e^2)}= 0\,, 
\]
and  $(\id, M_f)$ has the (BE)-property.\\
To see that also de the pair $(M_f, M_0)$ has the property we have to see first that it satisfies \ac\!. We consider the matrix $ M_{0} M_{f}^{-1} = C = (c_{in})_{in}$, and note that $c_{ni}=0$ for $i>n$. For the remaining cases we recall \eqref{12473BF} and see that 
\[ 
c_{nn}= \frac{S(n,f)}{nf(n)} \,,
\]
and 
\[ 
c_{ni}= \frac{S(i,f)}{nf(i)}- \frac{S(i,f)}{nf(i+1)} \,,
\]
for $i\leq n-1$. Taking all these into account and the  fact that $ f$ is increasing,  we have, for each fixed $ n \in \mathbb{N}$
\begin{multline*}
\sum_{i=1}^\infty|c_{ni}|= \sum_{i=1}^n c_{ni}=  \frac{S(n,f)}{nf(n)} + \sum_{i=1}^{n-1} \bigg( \frac{S(i,f)}{nf(i)}- \frac{S(i,f)}{nf(i+1)} \bigg) \\
= \frac{S(n,f)}{nf(n)} + \sum_{i=1}^{n-1} \frac{S(i,f)}{nf(i)} - \sum_{i=2}^{n} \frac{S(i-1,f)}{nf(i)}\\
=\frac{1}{n} \bigg( \frac{S(1,f)}{f(1)}  +\sum_{i=2}^{n} \frac{f(i)}{f(i)}  \bigg)=1\,,
\end{multline*}
and the claim follows. \\
It remains to see that $(M_f,M_0\Delta)$ satisfies \accc\!. Observe that the limit in \eqref{lim sf} holds for $f$ using \eqref{ineqg}
\begin{multline*}
\lim_{n\to \infty} \frac{1}{n-1}\frac{S(n-1,f)}{f(n)} + \frac{1}{n}\frac{S(n,f)}{f(n)}\leq \lim_{n\to \infty} \frac{2}{n-1}\frac{S(n,f)}{f(n)} \\
 \leq  \lim_{n\to \infty} \frac{2}{n-1}\frac{\int_{1}^{n+1}f(x)\dx}{f(n)} =
\lim_{n\to \infty} \frac{e^{2\sqrt{n+1}}-e^2}{n-1}\frac{2\sqrt{n}}{e^{2\sqrt{n}}}\\ 
\leq \lim_{n\to \infty} \frac{2\sqrt{n}}{n-1}e^{2(\sqrt{n+1}-\sqrt{n})}=
\lim_{n\to \infty} \frac{2\sqrt{n}}{n-1}e^{\frac{2}{\sqrt{n+1}+\sqrt{n}}}=0\,.
\end{multline*}
This gives that $(M_f,M_0\Delta)$ satisfies \accc and, then $(M_f, M_0)$ has the (BE)-property.
\end{nota}

%

\begin{thebibliography}{10}

\bibitem{ABR}
A.~A. Albanese, J.~Bonet, and W.~J. Ricker.
\newblock Mean ergodic operators in {F}r\'echet spaces.
\newblock {\em Ann. Acad. Sci. Fenn. Math.}, 34(2):401--436, 2009.

\bibitem{AlbaneseBonetRicker15}
A.~A. Albanese, J.~Bonet, and W.~J. Ricker.
\newblock Spectrum and compactness of the {C}es\`aro operator on weighted
  {$\ell_p$} spaces.
\newblock {\em J. Aust. Math. Soc.}, 99(3):287--314, 2015.

\bibitem{BoPaRi11}
J.~Bonet, B.~de~Pagter, and W.~J. Ricker.
\newblock Mean ergodic operators and reflexive {F}r\'echet lattices.
\newblock {\em Proc. Roy. Soc. Edinburgh Sect. A}, 141(5):897--920, 2011.

\bibitem{llibre}
J.~Bonet, D.~Jornet, and P.~Sevilla-Peris.
\newblock {\em Function spaces and operators between them}, volume~11 of {\em
  RSME Springer Series}.
\newblock Springer, Cham, [2023] \copyright 2023.

\bibitem{boos}
J.~Boos.
\newblock {\em Classical and modern methods in summability}.
\newblock Oxford Mathematical Monographs. Oxford University Press, Oxford,
  2000.
\newblock Assisted by Peter Cass, Oxford Science Publications.

\bibitem{cooke_65}
R.~G. Cooke.
\newblock {\em Infinite matrices and sequence spaces}.
\newblock Dover Publications, Inc., New York, 1965.

\bibitem{Diestel}
J.~Diestel.
\newblock {\em Sequences and series in {B}anach spaces}, volume~92 of {\em
  Graduate Texts in Mathematics}.
\newblock Springer-Verlag, New York, 1984.

\bibitem{dunfordschwartz_88}
N.~Dunford and J.~T. Schwartz.
\newblock {\em Linear operators. {P}art {I}}.
\newblock Wiley Classics Library. John Wiley \& Sons, Inc., New York, 1988.
\newblock General theory, With the assistance of William G. Bade and Robert G.
  Bartle, Reprint of the 1958 original, A Wiley-Interscience Publication.

\bibitem{hardy_49}
G.~H. Hardy.
\newblock {\em Divergent {S}eries}.
\newblock Oxford, at the Clarendon Press,, 1949.

\bibitem{K2}
G.~K\"othe.
\newblock {\em Topological vector spaces. {II}}, volume 237 of {\em Grundlehren
  der Mathematischen Wissenschaften}.
\newblock Springer-Verlag, New York-Berlin, 1979.

\bibitem{krengel}
U.~Krengel.
\newblock {\em Ergodic theorems}, volume~6 of {\em De Gruyter Studies in
  Mathematics}.
\newblock Walter de Gruyter \& Co., Berlin, 1985.
\newblock With a supplement by Antoine Brunel.

\bibitem{Reade}
J.~B. Reade.
\newblock On the spectrum of the {C}es\`aro operator.
\newblock {\em Bull. London Math. Soc.}, 17(3):263--267, 1985.

\bibitem{Stieglitz}
M.~Stieglitz and H.~Tietz.
\newblock Matrixtransformationen von {F}olgenr\"aumen. {E}ine
  {E}rgebnis\"ubersicht.
\newblock {\em Math. Z.}, 154(1):1--16, 1977.

\bibitem{Taylor}
A.~E. Taylor.
\newblock {\em Introduction to functional analysis}.
\newblock John Wiley \& Sons, Inc., New York; Chapman \& Hall, Ltd., London,
  1958.

\bibitem{Toeplitz1911}
O.~Toeplitz.
\newblock Über die lineare mittelbildungen.
\newblock {\em Prace Matematyczno-Fizyczne}, 22:113--119, 1911.

\bibitem{yosida1938}
K.~{Yosida}.
\newblock {Mean ergodic theorem in Banach spaces}.
\newblock {\em {Proc. Imp. Acad. Japan}}, 14:292--294, 1938.

\end{thebibliography}

\end{document}